\documentclass[11pt,oneside]{amsart}
\usepackage{ae,aecompl}
\usepackage[T1]{fontenc}
\usepackage[latin9]{inputenc}
\usepackage{color}
\usepackage{textcomp}
\usepackage{amstext}
\usepackage{amsthm}
\usepackage{amssymb}
\usepackage{comment}
\usepackage{amsmath,amscd}
\usepackage{graphicx}

\makeatletter
\numberwithin{equation}{section}
\numberwithin{figure}{section}
\theoremstyle{plain}
\newtheorem{thm}{\protect\theoremname}
\theoremstyle{plain}
\newtheorem{lem}[thm]{\protect\lemmaname}
\theoremstyle{remark}
\newtheorem{rem}[thm]{\protect\remarkname}
\theoremstyle{plain}
\newtheorem{prop}[thm]{\protect\propositionname}
\theoremstyle{definition}
\newtheorem{defn}[thm]{\protect\definitionname}
\theoremstyle{plain}
\newtheorem{cor}[thm]{\protect\corollaryname}
\newtheorem*{theorem*}{Theorem}
\newtheorem*{lem*}{Lemma}
\numberwithin{thm}{section}

\makeatother

\newtheorem{manualtheoreminner}{Theorem}
\newenvironment{manualtheorem}[1]{%
	\renewcommand\themanualtheoreminner{#1}%
	\manualtheoreminner
}{\endmanualtheoreminner}

\providecommand{\corollaryname}{Corollary}
\providecommand{\definitionname}{Definition}
\providecommand{\lemmaname}{Lemma}
\providecommand{\propositionname}{Proposition}
\providecommand{\remarkname}{Remark}
\providecommand{\theoremname}{Theorem}

\begin{document}
	\global\long\def\Exp#1{\mathcal{E}(#1)}%
	\global\long\def\T{T_{*}}%
	
	\global\long\def\hT{\widehat{T_{*}}}%
	
	\global\long\def\ZZ{\mathbb{Z}}
	
	\global\long\def\Autt{\text{Aut}_{\,2}\left(X,\mathcal{A},\mu\right)}
	
	\global\long\def\NN{\mathbb{N}}
	
	\global\long\def\Auto{\text{Aut}_{\,\hat{0}}\left(X,\mathcal{A},\mu\right)}
	
	\global\long\def\Autn{\text{Aut}_{\,\mathrm{NMP}}\left(X,\mathcal{A},\mu\right)}
	
	\global\long\def\EAutn{\text{EAut}_{\,\mathrm{NMP}}\left(X,\mathcal{A},\mu\right)}
	
	\global\long\def\Auti{\text{Aut}_{\,1}\left(X,\mathcal{A},\mu\right)}
	
	\title{Nonsingular Poisson Suspensions}

	\author{Alexandre I. Danilenko}
	\address{B. Verkin Institute for Low Temperature Physics and Engineering of National Academy of
		Sciences of Ukraine, 47 Nauky Ave., Kharkiv, 61103, Ukraine}
	\email{alexandre.danilenko@gmail.com}
	\author{Zemer Kosloff}
	\thanks{The research of Z.K. was partially supported by ISF grant No. 1570/17}
	\address{Einstein Institute of Mathematics,
		Hebrew University of Jerusalem, Edmond J. Safra Campus, Jerusalem 91904,
		Israel}
	\email{zemer.kosloff@mail.huji.ac.il}
	\author{Emmanuel Roy}
	\address{Laboratoire Analyse, G\'eom\'etrie et Applications, CNRS UMR 7539,
		Universit\'e Paris
		13, Institut Galil\'ee, 99 avenue Jean-Baptiste Cl\'ement F93430 Villetaneuse, France.}
	\email{roy@math.univ-paris13.fr}
	
	\begin{abstract} The classical Poisson functor associates to every infinite measure preserving dynamical system $(X,\mu,T)$ a probability preserving dynamical system $(X^*,\mu^*,T_*)$ called the Poisson suspension of $T$.
		In this paper we generalize this construction: a subgroup Aut$_2(X,\mu)$ of $\mu$-nonsingular transformations $T$ of $X$ is specified as the largest subgroup for which $T_*$ is  $\mu^*$-nonsingular.
		Topological structure of this subgroup is studied.
		We show that a generic element in Aut$_2(X,\mu)$ is ergodic and of Krieger type III$_1$.
		Let   $G$ be a locally compact Polish group and let  $A:G\to\text{Aut}_2(X,\mu)$ be a $G$-action.
		We investigate dynamical properties of the Poisson suspension $A_*$ of $A$ in terms of an affine representation  of $G$ associated naturally with $A$.
		It is shown that $G$ has property (T) if and only if each nonsingular Poisson $G$-action admits an absolutely continuous invariant probability.
		If $G$ does  not have property $(T)$ then for each generating probability $\kappa$ on $G$ and $t>0$, a nonsingular Poisson $G$-action is constructed whose Furstenberg $\kappa$-entropy is $t$. 
	\end{abstract}
	
	\maketitle
	
	\section{Introduction}
	\subsection{Poisson suspensions: measure preserving and nonsingular}
	In this  paper we initiate a systematic study  of {\it nonsingular} Poisson suspensions in the framework of ergodic theory.
	Poisson point processes 
	have convenient mathematical properties and are often used as mathematical models for seemingly random phenomena.
	Say, in statistical physics they provide a model for ideal gas consisting of countably many randomly moving noninteracting points (particles) of a standard $\sigma$-finite infinite nonatomic measure space $(X,\mathcal{A},\mu)$.
	The space of states $X^*$ of the gas consists of countable subsets (configurations) of $X$.
	It is endowed with the natural Borel $\sigma$-algebra $\mathcal A^*$ and a probability measure $\mu^*$ such that for each subset $A\in\mathcal A$, the $\mu^*$-expected value of the number of particles in $A$ has the Poisson distribution with parameter $\mu(A)$.
	Given a transformation $T:X\to X$,  
	the Poisson suspension $T_*:X^*\to X^*$ of $T$ models the motion of the particles by $T$. 
	In other words,  $T_*\omega=\{Tx:\,x\in\omega\}$ for all $\omega\in X^*$.
	If $\mu^*\circ T_*\sim\mu^*$ then we call the system $(X^*,\mathcal A^*,\mu^*,T_*)$  {\it the nonsingular Poisson suspension} of $(X,\mathcal A,\mu,T)$.

	In ergodic theory, only the {\it measure preserving} case $\mu\circ T=\mu$ (and hence $\mu^*\circ T_*=\mu^*$) of Poisson suspensions has been considered so far. 
	We refer the reader to the classical sources \cite{Cor82ergo}  studying the dynamical properties of the ideal gas and to \cite{OrnWeis87Entr} providing a Poisson model for Bernoulli actions of 
	locally compact groups\footnote{The seemingly simpler Bernoulli model as the shiftwise $G$-action on the product space $A^G$, for a probability space $A$, drops out from the category of standard measure spaces if the group $G$ is not countable.}.
	Over the last 15 years we observe a boost of interest to Poisson suspensions in ergodic theory.
	The work \cite{Roy07Infinite} describes  dynamical properties of $T_*$ such as ergodicity, weak mixing, mixing, rigidity, $K$-property, etc. in terms of $T$ and $\mu$.
	Spectral properties, entropic properties, similarity and  asymmetry, properties of joinings of the measure preserving Poisson suspensions are studied extensively in   \cite{Roy07Infinite}, \cite[\S 8]{Dan}, \cite{Jan08entropy}, \cite{DaRy}, \cite{Lem05ELF}, \cite{JanRueRoy15Sushis}, etc.
	Summarizing this progress we can say that  the {\it Poisson functor} $T\to T_*$ from the category of infinite measure preserving actions to the probability preserving actions is similar (and of similar importance in ergodic theory) to the {\it Gaussian functor} from the category of orthogonal representations to the probability preserving actions.

	Our global task is to find  some extensions of  the aforementioned results to the {\it nonsingular} Poisson suspensions  as well as to investigate purely nonsingular properties of them such as dissipativeness, 
	Krieger's type, associated flow, Furstenberg entropy.  
	We note that nonsingular Poisson suspensions are widely used in  
	the representation theory to construct unitary representations of {\it large groups} such as diffeomorphism groups of  non-compact manifolds and   current groups (see
	the surveys \cite{VershikGraev}, \cite{GelGraVer} and references therein and Chapter~10 of the book
	\cite{Ner1996CatInfGroups}).
	The fundamental (in fact, the only dynamical) property of the nonsingular Poisson suspensions of these group actions exploited there is the ergodicity, because it implies irreducibility of the associated unitary representation.
	In contrast to that, in this work we investigate  dynamical properties of nonsingular Poisson suspensions for individual transformations or, more generally, {\it locally compact} group actions.
	For them, the aforementioned implication does not hold.

	\subsection{Main results} We now list the main  results of this paper.
	The following fact can be deduced easily from Takahashi's version of Kakutani dichotomy for Poisson point processes:\\
	
	The set  of $\mu$-nonsingular transformations $T$ of $(X,\mathcal A,\mu)$ such that $T_*$ is $\mu^*$-nonsingular is exactly the group
	$$
	\text{{\rm Aut}}_2(X,\mathcal A,\mu):=\bigg\{T: \mu\circ T^{-1}\sim \mu, \ \sqrt{\frac{d\mu\circ T^{-1}}{d\mu}}-1\in L^2(\mu)\bigg \}.
	$$
	
	We note that Aut$_2(X,\mathcal A,\mu)$ contains  (properly) a subgroup 
	$$
	\text{Aut}_1(X,\mathcal A,\mu):=
	\bigg\{T: \mu\circ T^{-1}\sim \mu, \ {\frac{d\mu\circ T^{-1}}{d\mu}}-1\in L^1(\mu)\bigg\},
	$$
	which is the largest subgroup of nonsingular transformations for which the Poisson suspensions were defined in the literature  (\cite{VershikGraev}, \cite{GelGraVer}, 
	\cite{Ner1996CatInfGroups}\footnote{Neretin's notation for Aut$_1(X,\mathcal A,\mu)$ is 
		Gms$_\infty$. }) so far.
	Let   $\mathcal U_\mathbb{R}(L^2(\mu))$ denote the group of unitary operators in $L^2(X,\mu)$ preserving the real functions and let $\text{Aff}_\mathbb{R}(L^2(\mu)):=L^2_\mathbb{R}(X,\mu)\rtimes
	\mathcal U_\mathbb{R}(L^2(\mu))$ denote the group of affine operators in $L^2(X,\mu)$ preserving the real functions.
	We consider two natural representations of Aut$_2(X,\mathcal A,\mu)$ in $L^2(X,\mu)$: the {\it unitary} (well known) Koopman representation
	$
	U:\text{Aut}_2(X,\mathcal A,\mu)\ni T\mapsto U_T\in \mathcal U_\mathbb{R}(L^2(\mu))
	$
	and the  {\it affine} representation $A^{(2)}:\text{Aut}_2(X,\mathcal A,\mu)\ni T\mapsto A^{(2)}_T\in  \text{Aff}_\mathbb{R}(L^2(\mu))$, given by the formulas
	\begin{equation}\label{eq:un-aff}
	U_Tf:=f\circ T^{-1}\sqrt{\frac{d\mu\circ T^{-1}}{d\mu}}\quad\text{and}\quad A^{(2)}_Tf:= U_Tf+\sqrt{\frac{d\mu\circ T^{-1}}{d\mu}}-1.
	\end{equation}
	It appears, surprisingly for us, that the nonsingular Poisson suspensions are related closely to geometrical (affine) properties of the underlying Hilbert space.
	
	\begin{manualtheorem}{A}
		Let $C_{-1}:=\{f\in L^2(\mu):f\ge-1\}$.
		Then $\{A\in \text{{\rm Aff}}_\mathbb{R}(L^2(\mu)): AC_{-1}=C_{-1}\}=\Big\{\Big(\sqrt{\frac{d\mu\circ T^{-1}}{d\mu}}-1, U_T\Big):\,T\in \text{{\rm Aut}}_2(X,\mathcal A,\mu)\Big\}$.
	\end{manualtheorem}
	
	A similar results holds also for Aut$_1(X,\mathcal A,\mu)$ and  the corresponding natural isometric and affine representations of this group in $L^1(\mu)$.
	
	Let
	$
	W:\text{ Aff}_\mathbb{R} R(L^2(\mu))\ni A\mapsto  W_A\in \mathcal U(F(L^2(\mu)))
	$
	stand for the  well known unitary representation 
	of $\text{ Aff}_\mathbb{R}(L^2(\mu))$ by the Weyl operators $W_A$ in the Fock space $F(L^2(\mu))$ constructed over $L^2(\mu)$ \cite{Gui72Sym}.
	The operators $W_A$ are of fundamental importance in representation theory and quantum probability.
	They also appear naturally in description of  the unitary Koopman operators associated with  nonsingular Poisson suspensions.
	
	\begin{manualtheorem}{B}
		
		Under the natural identification of $L^2(\mu^*)$ with $F(L^2(\mu))$, for each transformation $T\in\text{{\rm Aut}}_2(X,\mathcal A,\mu)$, the Koopman operator $U_{T_*}$ generated by $T_*$ equals $W_{A^{(2)}_T}$.
	\end{manualtheorem}
	
	Theorem~A enables us to define a Polish topology, denoted by $d_2$, on Aut$_2(X,\mathcal A,\mu)$.
	In a similar way, utilizing the aforementioned analogue of Theorem~A for Aut$_1(X,\mathcal A,\mu)$ we also introduce a Polish topology $d_1$ on Aut$_1(X,\mathcal A,\mu)$.
	We show that the homomorphism $\chi:\text{Aut}_1(X,\mathcal A,\mu)\to\mathbb{R}$, defined in 
	\cite{Ner1996CatInfGroups} by the formula $\chi(T):=\int_X\left(\frac{d\mu\circ T^{-1}}{d\mu}-1\right)d\mu$, is $d_1$-continuous.
	We then prove the following results.
	
	\begin{manualtheorem}{C}
		\begin{itemize}
			\item
			The weak topology is strictly weaker than $d_2$ and $d_2$ is strictly weaker than $d_1$.
			\item
			{\rm Aut}$_p(X,\mu)$ endowed with $d_p$ is a Polish group for  $p=1,2$.
			\item
			The set $\{T_*:\, T\in\text{{\rm Aut}}_2(X,\mathcal A,\mu)\}$ of nonsingular Poisson suspensions is  a weakly closed  subgroup of  nonsingular transformations of $(X^*,\mathcal A^*,\mu^*)$.
			\item
			\text{{\rm Aut}}$_1(X,\mathcal A,\mu)$ is isomorphic to a semidirect product
			{\rm Ker}$\,\chi\rtimes\mathbb{R}$ is such a way that $\chi$ corresponds to the projection onto the second coordinate.
			\item Every conservative transformation in \text{{\rm Aut}}$_1(X,\mathcal A,\mu)$ belongs to 
			{\rm Ker}$\,\chi$.
		\end{itemize}
	\end{manualtheorem}
	
	The latter result (as well as the following theorem) shows that the group {\rm Ker}$\,\chi$ is a more natural
	object than \text{{\rm Aut}}$_1(X,\mathcal A,\mu)$ from the ergodic theory point of view.

	\begin{manualtheorem}{D}
		\begin{itemize}
			\item {\rm Aut}$_2(X,\mathcal A,\mu)$ endowed with $d_2$ has the Rokhlin property\footnote{I.e. there is a dense conjugacy class in this group.}.
			The subset of ergodic transformations  of Krieger's type $III_1$ is a dense $G_\delta$ in
			{\rm Aut}$_2(X,\mu)$.
			\item {\rm Ker\,}$\chi$ endowed with $d_1$ has the Rokhlin property ({\rm Aut}$_1(X,\mathcal A,\mu)$ does not have it).
			The subset of ergodic transformations  of Krieger's type $III_1$ is a dense $G_\delta$ in
			{\rm Ker\,}$\chi$.
		\end{itemize}
	\end{manualtheorem}

	It is well known that the group of all $\mu$-nonsingular transformations of $(X,\mathcal A,\mu)$ endowed with the weak topology also possesses  similar properties (see \cite{IonTul65}, \cite{ChKa79}, \cite{ChHaPr87}, \cite{DaSi}).
	However  the proof of  Theorem~D is more difficult for several reasons.
	The first is that  $d_2$ and $d_1$  are stronger than the  weak topology.
	The second is that we have no freedom to replace $\mu$ by an arbitrary equivalent measure any more.
	Indeed, if $\nu\sim\mu$ but $\sqrt{\frac{d\nu}{d\mu}}-1\not\in L^2(\mu)$ then 
	$\text{{\rm Aut}}_2(X,\mathcal A,\nu)\ne \text{{\rm Aut}}_2(X,\mathcal A,\mu)$.
	
	In a subsequent work \cite{DanKosRoy} we show that the subset of $T\in \text{Aut}_2(X,\mathcal{A},\mu)$ such that $T_*$ is ergodic and type $\mathrm{III}_1$ is a dense $G_\delta$.
	Combined with Theorem~D this implies the existence of a type $\mathrm{III}_1$  ergodic transformation whose Poisson suspension is also ergodic and of type $\mathrm{III}_1$.
	
	Let $G$ be a locally compact second countable group.
	The affine isometric representations  of $G$ is a fundamental tool in geometric group theory connecting Kazhdan property (T), the Haagerup property, operator algebras, harmonic analysis, etc. (see \cite{BePiVa},
	\cite{CoTeVa1}, \cite{CoTeVa2}, \cite{Bekka}).
	The Poisson $G$-actions deliver  natural non-trivial  examples of such representations.
	Indeed, if $T:G\ni g\mapsto T_g\in \text{Aut}_2(X,\mathcal A,\mu)$ is a measurable $G$-action  then the mapping  $G\ni g\mapsto A^{(2)}_{T_g}$ (see (\ref{eq:un-aff})) 
	is a continuous  affine representation of $G$ in $L^2(\mu)$.

	\begin{manualtheorem}{E}
		
		Let  $T:G\ni g\mapsto T_g\in \text{{\rm Aut}}_2(X,\mathcal A,\mu)$ be a measurable $G$-action.
		\begin{itemize}
			\item
			The Poisson suspension $ T_*:=\{(T_g)_*\}_{g\in G}$ of $T$ has an absolutely continuous invariant probability measure if and only if the $L^2(\mu)$-cocycle $c_T:G\ni g\mapsto c_T(g):=
			\sqrt{\frac{d\mu\circ T_g^{-1}}{d\mu}}-1$ is bounded.
			\item
			If $
			\int_G e^{-\frac12\|c_T(g)\|_2^2} \,d\lambda(g)<\infty
			$
			then $T_*$ is totally dissipative.
			\item
			$T_*$ is of zero type if and only if $\|c_T(g)\|_2\to\infty$ as $g\to\infty$.
		\end{itemize}
	\end{manualtheorem}
	
	The aforementioned {\it zero type} is a nonsingular analogue of the mixing in the probability preserving case (see \cite{DaSi}).
	We could not find a criterion for ergodicity of $T_*$ because it depends in a subtle way not only on the ergodic properties of $T$ but also on a ``right'' choice of a measure inside the equivalence class of $\mu$: we construct an example of an ergodic $T$ admitting a $\mu$-equivalent invariant probability such that $T_*$ is totally dissipative.
	
	Let $\kappa$ be a  probability on $G$.
	A nonsingular $G$-action $S=(S_g)_{g\in G}$ on a probability space $(Y,\mathcal B,\nu)$ is called
	$\kappa$-stationary if $\int_G\kappa(g)\nu\circ S_g\,d\kappa(g)=\nu$ (see a survey
	\cite{FuGl} for properties and applications of the stationary actions).
	
	\begin{manualtheorem}{F}
		
		Let $T$ be as in the previous theorem.  
		If $T_*$ is $\kappa$-stationary for a generating probability $\kappa$ on $G$ then $T_*$ preserves $\mu^*$ and $T$ preserves $\mu$.
	\end{manualtheorem}  
	
	There are several equivalent characterizations of  property (T) for $G$~\cite{BHV08Kazh}.
	We provide one more in terms of the nonsingular Poisson suspensions.
	
	\begin{manualtheorem}{G}
		
		$G$ has property  (T) if and only if each nonsingular Poisson $G$-action $T_*$ 
		admits an absolutely continuous invariant probability.
	\end{manualtheorem}
	
	One more characterization was obtained in \cite{BowHarTam}: if a countable discrete group $G$ does not have property (T) then for each generating probability $\kappa$ on $G$, the Furstenberg $\kappa$-entropy $h_\kappa(.)$   {\it has no  gap} on the set of purely infinite ergodic nonsingular $G$-actions.
	This was refined in \cite{Danilen}: $h_\kappa(.)$ takes all possible positive values on the subset
	of ergodic $G$-actions of type $III_1$.
	We extend this result to arbitrary locally compact groups and partly refine it by considering only
	the nonsingular Poisson suspensions.
	
	\begin{manualtheorem}{H}
		Let a locally compact $G$ do not have property (T) and let $\kappa$ be a probability  on $G$.
		Then there is a nonsingular $G$-action $T$ on an infinite measure space $(X,\mu)$
		such that the Poisson suspension $T_*$ of $T$ is $\mu^*$-nonsingular and 
		$\{h_\kappa(T_*,\mu_t^*): t\in(0,+\infty)\}=(0,+\infty)$, where $\mu_t=t\mu$. 
	\end{manualtheorem}

	During the course of work on this paper we learnt about   \cite{ArIsMa}, where
	it was constructed a functor from the affine  $G$-representations
	to the  nonsingular Gaussian systems.
	Certain nonsingular Gaussian counterparts of Theorems~G and H are proved there.
	We believe that there should be some interplay between the theory of nonsingular Gaussian actions and the nonsingular Poisson suspensions.

	\subsection{Sections overview}
	In Section \ref{sec: Poisson suspensions} we introduce a model $(X^*,\mathcal A^*,\mu^*, T_*)$ for the Poisson suspension
	of   a dynamical system $(X,\mathcal A,\mu,T)$, consider $L^2(\mu^*)$ as a Fock space over $L^2(\mu)$ and extend
	the exponential map $\mathcal E:L^2(\mu)\to L^2(\mu^*)$ to some non-square integrable functions.
	In Section~\ref{sec: AC of Poisson} we extend and refine Takahashi's theorem on equivalence and orthogonality of the Poisson suspensions of equivalent measures \cite{Tak90}.
	In Section  \ref{sec: NS of Poisson and Koopman} we introduce and study  the topological groups Aut$_2(X,\mathcal A,\mu)$,  Aut$_1(X,\mathcal A,\mu)$, the unitary and affine representations $U$ and $A^{(2)}$ of Aut$_2(X,\mathcal A,\mu)$ and their analogues for Aut$_1(X,\mathcal A,\mu)$, and prove Theorems~A, B, C and related results.
	Section~\ref{sec: Baire category} is devoted to generic properties of Aut$_2(X,\mathcal A,\mu)$ and Aut$_1(X,\mathcal A,\mu)$. 
	Theorem~D is proved there.
	In Section~\ref{sec: Basic dynamics}, for arbitrary locally compact groups $G$, we characterize some basic dynamical properties of the nonsingular Poisson suspensions  $T_*$ of $G$-actions $T$ in terms of the underlying system $(X,\mu,T)$.
	Theorem~E is proved there. 
	In Section~\ref{sec: F-entropy} we consider stationary $G$-actions, prove Theorem~F and compute the Furstenberg $\kappa$-entropy of $T_*$ in terms of the underlying system $(X,\mu,T)$.
	Section~\ref{sec: Property T} is devoted to property (T).
	We prove Theorems~G and H and related results there.
	The paper ends with Appendix which is devoted to infinitely divisible variables and stochastic integration. 
	This material is used in the course of the proofs of Theorems~\ref{th:R-N-D} and
	\ref{thm: LogT'}.

\section{Poisson suspensions}\label{sec: Poisson suspensions}
\subsection{Space of point processes}\label{space}
Let $(X,\mathcal{A},\mu)$ be a $\sigma$-finite Lebesgue
space with a non-atomic measure, that is $\left(X,\mathcal{A},\mu\right)$
is $\bmod  0$ isomorphic to the real line if $\mu(X)=\infty$
or to a bounded closed interval if $\mu(X)<\infty$, endowed
with Lebesgue measure and Lebesgue measurable sets.

The space of point processes over $X$ is defined as the set $X^{*}$
of all measures $\omega$ of the form $\omega=\sum_{i\in I}\delta_{x_{i}}$
where $I$ is at most countable. $X^{*}$ is endowed with the smallest
$\sigma$-algebra such that the $\Bbb Z_+\cup\{+\infty\}$-valued maps
$N_{A}:\,\omega\mapsto\omega\left(A\right)$ are measurable, for all
$A\in\mathcal{A}$.
We denote this $\sigma$-algebra on $X^{*}$ by $\mathcal{A}^{*}$.

\subsection{Poisson measures}

Let $\mathcal{A}_{f}^{\mu}$ denote the collection of sets $A\in\mathcal{A}$
of finite $\mu$-measure. There exists a unique probability measure
$\mu^{*}$ on $\left(X^{*},\mathcal{A}^{*}\right)$ such that:
\begin{itemize}
\item For all $k\ge1$ and pairwise disjoint sets $A_{1},\dots,A_{k}$ in
$\mathcal{A}_{f}^{\mu}$, the random variables $N_{A_{i}}$, $1\le i\le k$,
are independent.
\item For any $A\in\mathcal{A}_{f}^{\mu}$, $N_{A}$ is Poisson distributed
with parameter $\mu\left(A\right)$.
\end{itemize}
The probability space $\left(X^{*},\mathcal{A}^{*},\mu^{*}\right)$
is called the \emph{Poisson space} over the base $\left(X,\mathcal{A},\mu\right)$.
When completed with respect to $\mu^{*}$, $\left(X^{*},\mathcal{A}^{*},\mu^{*}\right)$
is a Lebesgue space.
The random measure $A\mapsto N_{A}$, $A\in\mathcal{A}$ distributed
as $\mu^{*}$ is called a \emph{Poisson point process of intensity
$\mu$}. In most cases this object is presented on  $\mathbb{R}^{d}$
(or on a subset of it) with Lebesgue measure as intensity and then
called \emph{homogeneous Poisson point process}. As Lebesgue spaces
with a continuous measure are either isomorphic to $\mathbb{R}$ or
to a bounded closed interval with Lebesgue measure depending on wether
the measure is finite or not, there is essentially no loss in generality
in dealing with homogeneous Poisson point process.
Observe the following three important features of a Poisson measure:
\begin{itemize}
\item $\mu^{*}$ is supported on \emph{simple counting measures}, that is:
\[
\mu^{*}(\{\omega\in X^{*}:\forall x\in X, \:\omega(\{ x\} )=0\text{ or }1\})=1.
\]
\item The \emph{intensity} of the random measure $A\mapsto N_{A}$ is $\mu$,
that is $\mathbb{E}_{\mu^{*}}[N_{A}]=\mu(A)$.
\end{itemize}
We shall also make use of the following important theorem:
\begin{thm}
(R\'enyi) Let $m$ be a probability measure on $\left(X^{*},\mathcal{A}^{*}\right)$
supported on  simple counting measures.
If for all $A\in\mathcal{A}_{f}^{\mu}$,
$$
m(\{\omega: N_{A}(\omega)=0\}) =\mu^{*}(\{\omega: N_{A}(\omega)=0\}) ,
$$ 
then
$m=\mu^{*}$.
\end{thm}

\subsection{Poisson suspensions}

At the core of this paper is an easy yet fundamental observation:
If $\varphi$ is a measurable map between two Lebesgue spaces $\left(X,\mathcal{A},\mu\right)$
and $\left(Y,\mathcal{B},\nu\right)$ such that $\mu\circ\varphi^{-1}=\nu$
then $\varphi_{*}$ also acts measurably between $\left(X^{*},\mathcal{A}^{*}\right)$
and $\text{\ensuremath{\left(Y^{*},\mathcal{B}^{*}\right)} }$ by
$\varphi_{*}\omega=\omega\circ\text{\ensuremath{\varphi}}^{-1}$ and
satisfies
\[
\mu^{*}\circ\varphi_{*}^{-1}=\nu^{*}.
\]
In particular, if $T$ is an invertible transformation of $\left(X,\mathcal{A},\mu\right)$,
we have the following picture:
$$
\begin{aligned}
(X,\mathcal{A},\mu) & \overset{T}{\longrightarrow} 
(X,\mathcal{A},\mu\circ T^{-1})\\
(X^{*},\mathcal{A}^{*},\mu^{*})&\overset{T_*}\longrightarrow
(X^{*},\mathcal{A}^{*},\mu^{*}\circ T_{*}^{-1}).
\end{aligned}
$$
We will be interested in the situation where $T$ is a non-singular
automorphism, that is $\mu\sim\mu\circ T^{-1}$.
 It is not always
true that $\mu^{*}\sim\left(\mu\circ T^{-1}\right)^{*}$.
We will
recall necessary and sufficient conditions to get the equivalence
of measures. When it is the case, the non-singular dynamical system
$\left(X^{*},\mathcal{A}^{*},\mu^{*},T_{*}\right)$ will be called
the \emph{Poisson suspension} over $\left(X,\mathcal{A},\mu,T\right)$.

\subsection{Fock space structure of $L^{2}(\mu^{*})$ and coherent
vectors}\label{Fock}

We recall a very important structural feature of Poisson measures
(see \cite{Attal16Quant} or \cite{Ner1996CatInfGroups}):
there is a canonical isometry  between
$L^{2}\left(\mu^{*}\right)$ and  the \emph{symmetric
Fock space} $F(L^{2}(\mu))$ over $L^{2}(\mu)$.
We recall that
\[
F(L^{2}(\mu)):=\bigoplus_{n=0}^{\infty}L^{2}(\mu)^{\odot n},
\]
where $L^{2}(\mu)^{\odot0}:=\mathbb{C}$ and each factor
$L^{2}(\mu)^{\odot n}$ is equipped with the normalized
scalar product $n!\left\langle \cdot,\cdot\right\rangle _{L^{2}\left(\mu\right)^{\odot n}}$.
The Hilbert space $L^{2}\left(\mu\right)^{\odot n}$  considered  as a subspace of $L^{2}(\mu^{*})$
is called the \emph{chaos of order $n$}.
We now explain how to construct the canonical isometry.
For that, we choose a distinguished family of vectors
in $F(L^{2}(\mu))$, called the \emph{coherent
vectors}, defined, for $f\in L^{2}(\mu)$, by
\[
\Exp f:=\sum_{k=0}^{\infty}\frac{1}{n!}f^{\otimes n}\in F(L^{2}(\mu)).
\]
They form a total family in $F(L^{2}(\mu))$
 and satisfy the exponential relation
\begin{equation}
\left\langle \Exp f,\Exp g\right\rangle _{F\left(L^{2}\left(\mu\right)\right)}=e^{\left\langle f,g\right\rangle _{L^{2}\left(\mu\right)}}.\label{eq:exponentialrelation}
\end{equation}
 Denote by $\mathcal{B}_{0}(X)$ 
 the subspace of
 $L^{2}(\mu)$-functions
with finite $\mu$-measure support.
Then the family $\{\Exp f: f\in \mathcal{B}_{0}(X) \}$ is also  total 
in $F(L^{2}(\mu))$ (see e.g. \cite{Ner1996CatInfGroups}, where it is shown that even
a   subspace of finitely valued  functions from $\mathcal{B}_{0}(X)$
generates a total family in $F(L^{2}(\mu))$).
On the other hand, for $f\in \mathcal{B}_{0}(X)$, 
define a bounded function  $\exp (f)$  on $X^*$ by setting:

\begin{equation}
\exp (f)(\omega)=e^{-\int_{X}fd\mu}\prod_{\{x\in X:\,\omega(\{ x\}) =1\}}\left(1+f\left(x\right)\right),\quad\omega\in X^{*}.\label{eq:coherent_simple function-1}
\end{equation}
In particular,  for any set $A\in\mathcal{A}_{f}^{\mu}$,
\begin{equation}\label{eq:simple exp}
\exp{(-1_{A})}=e^{\mu\left(A\right)}1_{\{\omega:\, N_{A}(\omega)=0\} }.
\end{equation}
A standard calculation shows that 
\begin{equation}\label{eq:exp}
\left\langle \exp( f),\exp (g)\right\rangle _{L^{2}(\mu^*)}=e^{\left\langle f,g\right\rangle _{L^{2}\left(\mu\right)}}.
\end{equation}
Due to the R\'enyi Theorem, the family $\{\exp (f): f\in \mathcal{B}_{0}(X)\}$ is total
in $L^2(\mu^*)$.
Hence we deduce from (\ref{eq:exponentialrelation}) and  (\ref{eq:exp})
the map $\Exp f\mapsto\exp( f)$ extends to an isometry 
between  $F(L^{2}(\mu))$ and $L^2(\mu^*)$.
In the sequel, we will not distinguish between $\Exp f$ and  $\exp( f)$.
We will use the following properties of  coherent vectors: $\Exp{ \overline{f}}=\overline{\Exp f}$, $\Exp f\in L^1(\mu^*)$ and $\mathbb{E}_{\mu^{*}}[\Exp{f}]=1$ for all $f\in L^2(\mu)$.

\subsection{\label{subsec:Product-formula}Product formula and extended coherent
vectors}
For every two functions $f$ and $g$ in $L^2(\mu)$,
we define a function $f\bullet g$ by setting
$$
f\bullet g:=\left(1+f\right)\left(1+g\right)-1.
$$
We now define
 a function space
\[
\mathcal{L}\left(\mu\right):=\{ \varphi:X\to\mathbb{R}:\, \exists f,g\in L^{2}(\mu),\,\varphi=f\bullet g\}.
\]
Clearly, $L^{2}\left(\mu\right)\subset\mathcal{L}\left(\mu\right)$.
If $f,g\in \mathcal{B}_{0}(X)$,
 $f\bullet g$ is integrable
and has finite measure support.
Moreover, one can deduce from 
(\ref{eq:coherent_simple function-1})
 the following \emph{product
formula}:
\begin{equation}
\Exp f(\omega)\Exp g(\omega)=e^{\int_X fg\,d\mu}e^{-\int_{X}f\bullet g\,d\mu}\prod_{\{x\in X: \,\omega(\{ x\}) =1\}}\left(f\bullet g\left(x\right)+1\right).\label{eq:Product formula}
\end{equation}
This formula enables us to extend the definition of coherent vectors
to functions in $\mathcal{L}\left(\mu\right)$. 
Namely, we set
\begin{equation}\label{eq:product formula}
\Exp{f\bullet g}:=e^{-\int_{X}fgd\mu}\Exp f\Exp g
\end{equation}
 for all $f,g\in L^{2}(\mu)$.
 We have to  verify that this formula is well defined.
 To that end, we first
define an auxiliary  map $\Psi:L^{2}(\mu)\times L^{2}(\mu)\to L^{1}(\mu^{*})$
by setting 
\begin{equation}\label{eq:Psi}
\Psi(f,g):=e^{-\int fgd\mu}\Exp f\Exp g.
\end{equation}
Since the map $L^{2}(\mu)\ni f\mapsto\Exp f\in  L^2(\mu^*)$ is continuous, it follows that
$\Psi$ is continuous.
Now we consider $f,g,f',g'\in L^{2}(\mu)$  such that $f\bullet g=f^{\prime}\bullet g^{\prime}$.
 Select an increasing sequence $A_1\subset A_2\subset\cdots$ of subsets of finite measure in $X$ such that $\bigcup_{n=1}^\infty A_n=X$.
In we now set 
 $f_{n}:=f1_{A_{n}}$,  $g_{n}:=g1_{A_{n}}$,  $f_{n}':=f'1_{A_{n}}$ and $g_{n}':=g'1_{A_{n}}$
 then $f_n,g_n,f_n',g_n'\in\mathcal{B}_{0}(X)$ for each $n\in\Bbb N$ and
 $f_n\to f$, $g_n\to g$, $f_n'\to f'$, $g_n'\to g'$ in $L^2(\mu)$ as $n\to\infty$.
 Moreover, for all $n\in\mathbb{N}$, 
\[
f_{n}\bullet g_{n}=(f+g+fg)1_{A_n}=(f'+g'+f'g')1_{A_n}=f_{n}'\bullet g_{n}'.
\]
It now follows from (\ref{eq:Psi})  and  (\ref{eq:Product formula}) that 
$
\Psi\left(f_{n},g_{n}\right)=\Psi(f_{n}^{'},g_{n}^{'}).
$
Taking limits as $n\to\infty$ and using the continuity of $\Psi$, we
obtain that 
\[
\Psi(f,g)=\Psi(f',g'),
\]
as desired.
Thus, utilizing (\ref{eq:product formula}), we can define $\Exp\phi$ for each $\phi\in\mathcal{L}(\mu)$.
We call such  $\Exp\phi$ {\it the extended coherent vectors}.
If $\phi\in L^2(X)$ then $\phi=\phi\bullet 0$ and (\ref{eq:product formula}) implies that the extended coherent vector $\Exp\phi$ coincides with the ``standard'' coherent vector defined by $\phi$.
We also note that $\mathbb{E}_{\mu^{*}}(\Exp{\varphi})=1$ for  each   $\varphi\in\mathcal{L}(\mu)$.

\subsection{More properties of coherent vectors}
\begin{prop}
\label{lem:square coherent vectors}An extended coherent vector $\Exp{\varphi}$,
$\varphi\in\mathcal{L}(\mu)$, is in $L^{2}(\mu^{*})$
if and only if $\varphi\in L^{2}(\mu)$.
\end{prop}

\begin{proof}
If $\varphi\in L^2(\mu)$ then the extended coherent vector $\Exp{\varphi}$ is the classical coherent vector and hence it belongs to $L^2(\mu^*)$.

Now we prove the converse. 
Let $f,g\in L^{2}(\mu)$ and $\Exp{f\bullet g}\in L^2(\mu^*)$.
Denote by $\mathcal S$ the subspace of finitely valued functions from $L^2(\mu)$.
Then for each $h \in \mathcal S$, the function $g\bullet h$ is in $L^{2}(\mu)$.
We now have:
\begin{align*}
\mathbb{E}_{\mu^{*}}[\Exp{f\bullet g}\Exp h] & =\mathbb{E}_{\mu^{*}}[e^{-\int_{X}fg\,d\mu}\Exp f\Exp g\Exp h]\\
 & =e^{-\int_{X}fgd\mu+\int_{X}gh\,d\mu}\mathbb{E}_{\mu^{*}}[\Exp f\Exp{g\bullet h}]\\
 & =e^{-\int_{X}fg\,d\mu+\int_{X}gh\,d\mu+\int_{X}f\cdot (g\bullet h)\,d\mu}\\
 & =e^{\int_{X}(f\bullet g) \cdot h\,d\mu}
\end{align*}
Hence a linear functional $L:\mathcal S\ni h\mapsto\int_{X}(f\bullet g) \cdot h\,d\mu\in\Bbb C$ is well defined.
It is continuous at $0$.
Indeed, if a sequence $(h_n)_{n=1}^\infty$ with $h_n\in\mathcal S$, $n\in\Bbb N$, goes to 0 in $L^2(\mu)$ as $n\to\infty$ then  $\Exp {h_n}\to \Exp 0=1$ as $n\to\infty$ in $L^2(\mu^*)$.
Hence 
$$
\mathbb{E}_{\mu^{*}}[\Exp{f\bullet g}\Exp {h_n}]\to \mathbb{E}_{\mu^{*}}[\Exp{f\bullet g}]=1\quad\text{as }n\to\infty.
$$ 
Therefore $e^{\int_X(f\bullet g) \cdot h_n\,d\mu}\to 1$, i.e. $\int_{X}(f\bullet g) \cdot h_n\,d\mu\to 0$
as $n\to\infty$.
Since $L$ is linear, it follows that $L$ is continuous on the entire $\mathcal S$.
Since $\mathcal S$ is dense in $L^2(\mu)$, we deduce that $L$ extends uniquely  to a continuous linear functional on $L^2(\mu)$.
In view of the Riesz representation theorem, we conclude that $f\bullet g \in L^2(\mu)$.
\end{proof}

\begin{lem}
\label{lem:non-negative-coherent}
Let  a function $\varphi\in \mathcal{L}(\mu)$ take only real values.
Then the function $\widetilde{\varphi}:=\left|1+\varphi\right|-1$ belong to $\mathcal{L}(\mu)$
and 
\begin{equation}\label{eq:WW}
\left|\Exp{\varphi}\right|=e^{-2\int_{\{x\in X:\,\varphi(x)+1<0\}}\left(\varphi+1\right)\,d\mu}\Exp{\widetilde{\varphi}}.
\end{equation}
In particular, $\Exp{\varphi}$ is non-negative $\mu^{*}$-a.s. if
and only if $\varphi\ge-1$ $\mu$-almost everywhere.
\end{lem}

\begin{proof}
We consider separately three cases.
Suppose first that  $\varphi\in\mathcal{B}_{0}(X)$.
Then  $|1+\varphi|-1\in \mathcal{B}_{0}(X)$  and, in view of~(\ref{eq:coherent_simple function-1}),
\begin{equation}\label{eq:L}
\begin{aligned}
\left|\Exp{\varphi}\left(\omega\right)\right| & =e^{-\int_{X}\varphi \,d\mu}
\prod_{\{x\in X:\,\omega(\{ x\}) =1\}}|1+\varphi(x)|\\
 & =e^{-\int_{X}\varphi \,d\mu}\prod_{\{x\in X:\,\omega(\{ x\}) =1\}}(1+\widetilde\varphi(x))\\
 & =e^{-\int_{X}(\varphi -\widetilde\varphi) \,d\mu}\,e^{-\int_{X}\widetilde\varphi \,d\mu}
 \prod_{\{x\in X:\,\omega(\{ x\}) =1\}}(1+\widetilde\varphi(x))\\
  & =e^{-\int_{X}(\varphi -\widetilde\varphi) \,d\mu}\,
\Exp{\widetilde{\varphi}}(\omega)\\
 & =e^{-2\int_{\{x\in X:\,\varphi(x)+1 <0\}}(\varphi+1)\,d\mu}\Exp{\widetilde{\varphi}}(\omega),
\end{aligned}
\end{equation}
as desired.

Suppose now that $\varphi\in L^{2}(\mu)$.
We let $A_\varphi:=\{x\in X:\, \varphi(x)<-1\}$.
Then  $\mu(A_\varphi)<\infty$.
 Select a sequence $\{ \varphi_{n}\} _{n\in\mathbb{N}}$ of functions
$\varphi_n\in\mathcal{B}_{0}(X)$
such that 
\begin{itemize}
\item
$\varphi_n\to\varphi$ in 
 $L^{2}(\mu)$  as $n\to\infty$,
 \item
$A_{\varphi_n}= A_{\varphi}$ for each $n\in\Bbb N$ and
\item
$\varphi_n(x)=\varphi(x)$ if $x\in A_\varphi$.
 \end{itemize}
  It is straightforward to verify that  $|\widetilde \varphi|\le|\varphi|$ and 
 $|\widetilde{\varphi_{n}}-\widetilde{\varphi}|\le|\varphi_{n}-\varphi|$.
This yields  that $\widetilde\varphi\in L^{2}(\mu)$ and $\widetilde\varphi_n\to\widetilde\varphi$ 
 in $L^{2}(\mu)$
as $n\to\infty$.
Therefore
  $\Exp{\widetilde{\varphi_n}}\to
\Exp{\widetilde{\varphi}}$   in $L^2(\mu^*)$
 as $n\to\infty$.
By the first case  and the properties of $\varphi_n$,
$$
|\Exp{\varphi_n}|=e^{-2\int_{A_{\phi_n}}(\varphi_n+1)\,d\mu}\Exp{\widetilde\varphi_n}=
e^{-2\int_{A_{\varphi}}(\varphi+1)\,d\mu}\Exp{\widetilde\varphi_n}.
$$
Passing to the limit as $n\to\infty$, we obtain (\ref{eq:WW}), as desired.
Before we proceed to the general case, we  rewrite (\ref{eq:WW}) in the following equivalent form:
\begin{equation}\label{eq:QQ}
|\Exp{\varphi}|=e^{-\int_{X}(\varphi-\widetilde{\varphi})\,d\mu}\Exp{\widetilde{\varphi}}.
\end{equation}

Now, in the general case, let  $\varphi=f\bullet g$
for arbitrary vectors $f,g\in L^{2}(\mu)$.
A straightforward verification shows that $\widetilde \varphi=\widetilde f\bullet \widetilde g$.
We deduce from (\ref{eq:product formula}) and (\ref{eq:QQ}) that
\begin{align*}
\left|\Exp{\varphi}\right| 
 & =e^{-\int_{X}fg\,d\mu}\left|\Exp f\right|\left|\Exp g\right|\\
& =e^{-\int_{X}fg\,d\mu}\,e^{\int_X(\widetilde f-f)d\mu}\Exp{\widetilde{f}}\,
e^{\int_X(\widetilde g-g)d\mu}\Exp{\widetilde{g}}\\
 & =e^{-\int_{X}(fg -\widetilde f+f-\widetilde g+g-\widetilde f\widetilde g)d\mu}
 \Exp{\widetilde{f}\bullet\widetilde g}\\
 & =e^{-\int_X(\varphi-\widetilde\varphi)d\mu}\Exp{\widetilde{\varphi}},
\end{align*}
and (\ref{eq:WW}) follows.

To prove the second claim of the lemma, we assume first that $\Exp{\varphi}\ge 0$ $\mu^{*}$-a.s. for some  $\varphi\in\mathcal{L}(\mu)$.
Then
\begin{align*}
1 & =\mathbb{E}[\Exp{\varphi}]\\
 & =\mathbb{E}[|\Exp{\varphi}|]\\
 & =\mathbb{E}\left[e^{-2\int_{\{x\in X:\,\varphi(x)+1<0\}}(\varphi+1)\,d\mu}\Exp{\widetilde{\varphi}}\right]\\
 & =e^{-2\int_{\{x\in X:\,\varphi(x)+1<0\}}(\varphi+1)\,d\mu}.
\end{align*}
Therefore $\int_{\{x\in X:\,\varphi(x)+1<0\}}(\varphi+1)\,d\mu=0$, which
implies that $\varphi\ge-1$, $\mu$-almost everywhere.
Conversely, if $\varphi\ge-1$ $\mu$-almost everywhere then $\int_{\{x\in X:\,\varphi(x)+1<0\}}\left(\varphi+1\right)d\mu=0$ and $\widetilde\varphi=\varphi$.
Therefore $\left|\Exp{\varphi}\right|=\Exp{\widetilde{\varphi}}=\Exp{\varphi}$.
\end{proof}

\section{Absolute continuity and equivalence of Poisson measures}\label{sec: AC of Poisson}
Let $(X,\mathcal{A},\mu)$ be a non-atomic standard $\sigma$-finite measure space.
We single out four important sets of measures:
\begin{itemize}
\item $\mathcal{M}_{\mu,2}^{+}$ is the set of $\sigma$-finite measures
$\nu$ on $\left(X,\mathcal{A}\right)$ such that $\nu\ll\mu$ and
$\sqrt{\frac{d\nu}{d\mu}}-1\in L^{2}(\mu)$,
\item $\mathcal{M}_{\mu,2}^{\circ,+}:=\{\nu\in \mathcal{M}_{\mu,2}^{+}:\,
\nu\sim\mu\}$,
\item
$\mathcal{M}_{\mu,1}^{+}$ is the set of $\sigma$-finite measures
$\nu$ on $(X,\mathcal{A})$ such that $\nu\ll\mu$ and
$\frac{d\nu}{d\mu}-1\in L^{1}(\mu)$ and 
\item
$\mathcal{M}_{\mu,1}^{\circ,+}:=\{\nu\in \mathcal{M}_{\mu,1}^{+}:\,
\nu\sim\mu\}$.
\end{itemize}
Since
$
\left(\sqrt{x}-1\right)^{2}\le\left|x-1\right|
$
 for each $x>0$,
it follows that
 $\mathcal{M}_{\mu,1}^{+}\subset \mathcal{M}_{\mu,2}^{+}$ and  hence $\mathcal{M}_{\mu,1}^{\circ,+}\subset \mathcal{M}_{\mu,2}^{\circ,+}$.

\begin{rem}\label{rem: measures}
\begin{enumerate}
\item
\label{rem:finite measure} Let $\mu$ be a finite measure.
Then  $\nu \in \mathcal{M}_{\mu,2}^{+}$ if and only if it is a finite measure and $\nu\ll\mu$.
Hence $\mathcal{M}_{\mu,1}^{+}=\mathcal{M}_{\mu,2}^{+}$ and $\mathcal{M}_{\mu,1}^{\circ,+}=\mathcal{M}_{\mu,2}^{\circ,+}$.
\item
\label{rem:equivalence}If $\nu\in\mathcal{M}_{\mu,2}^{\circ,+}$,
then $\mathcal{M}_{\nu,2}^{\circ,+}=\mathcal{M}_{\mu,2}^{\circ,+}$.
In a similar way,  
if $\nu\in\mathcal{M}_{\mu,1}^{\circ,+}$,
then $\mathcal{M}_{\nu,1}^{\circ,+}=\mathcal{M}_{\mu,1}^{\circ,+}$.
\item
If $\mu$ is infinite and $\nu\in\mathcal{M}_{\mu,2}^{+}$ then $\mu\left(\left\{x\in X:\frac{d\nu}{d\mu}(x)=0\right\}\right)<\infty$.
\end{enumerate}
\end{rem}

\begin{lem}
\label{lem:A_f}Let $\nu\in\mathcal{M}_{\mu,2}^{+}$. Then for any
set $A\in\mathcal{A}$, we have that $\mu\left(A\right)<\infty$ if and only if
$\nu\left(A\right)<\infty$, that is $\mathcal{A}_{f}^{\mu}=\mathcal{A}_{f}^{\nu}$.
\end{lem}

\begin{proof}
We set $\phi:=\frac{d\nu}{d\mu}$. 
 If $\mu\left(A\right)<\infty$, then $\sqrt{\phi}-1\in L^{2}\left(\mu_{\mid A}\right)\subset L^{1}\left(\mu_{\mid A}\right)$.
Since $L^1(\mu)\ni\left(\sqrt{\phi}-1\right)^{2}=\left(\phi-1\right)-2\left(\sqrt{\phi}-1\right)$,
we now obtain that 
$\phi-1\in L^{1}\left(\mu_{\mid A}\right)$, which implies that
$\nu\left(A\right)<\infty$.

Now if $\nu\left(A\right)<\infty$ then  $\sqrt{\phi}$
is in $L^{2}\left(\mu_{\mid A}\right)$. As $\sqrt{\phi}-1\in L^{2}\left(\mu_{\mid A}\right)$,
this implies that the constant function $1$ is in $L^{2}\left(\mu_{\mid A}\right)$
too.
Hence $\mu\left(A\right)<\infty$.
\end{proof}
The following theorem is the ground for the rest of the paper. 
The necessary and sufficient condition for the absolute continuity for Poisson measures was found by Takahashi in  \cite{Tak90}.
However he did not write the explicit formula for the Radon-Nikodym derivative as a coherent vector. 
We only prove this formula and show that it generalizes the formula obtained by Neretin in \cite{Ner1996CatInfGroups} for the smaller class of measures $\mathcal{M}_{\mu,1}^{+}$.

\begin{thm}
\label{thm:absolute continuity}Let $\nu$ be a $\sigma$-finite measure
on $(X,\mathcal{A})$. 
Then $\nu^{*}\ll\mu^{*}$ if and
only if $\nu\in\mathcal{M}_{\mu,2}^{+}$.
If $\nu\in\mathcal{M}_{\mu,2}^{+}$ then
 $\frac{d\nu^{*}}{d\mu^{*}}=\Exp{\frac{d\nu}{d\mu}-1}$.
If $\nu\notin\mathcal{M}_{\mu,2}^{+}$, then $\nu^{*}\perp\mu^{*}$.
\end{thm}

\begin{proof}
Assume $\nu\in\mathcal{M}_{\mu,2}^{+}$ and set $\phi:=\frac{d\nu}{d\mu}$.
We now observe that  
\begin{equation}\label{eq:phi -1}
\phi-1=(\sqrt{\phi}-1)\bullet (\sqrt{\phi}-1)
\end{equation}
with $\sqrt{\phi}-1\in L^2(\mu)$.
Hence  $\Exp{\phi-1}$ is well defined.
Take a subset $A\in\mathcal{A}_{f}^{\nu}$.
Then  $A\in \mathcal{A}_{f}^{\mu}$
by Lemma \ref{lem:A_f}. 
Applying  the product formula~(\ref{eq:product formula}) three times, we obtain that
$$
\begin{aligned}
\Exp{-1_{A}}&\Exp{\phi-1}=\Exp{-1_{A}}e^{-\int_X(\sqrt\phi-1)^2d\mu}\Exp{\sqrt\phi-1}\Exp{\sqrt\phi-1}
\\
&=e^{-\int_X((\sqrt\phi-1)^2+(\sqrt\phi-1)1_A)d\mu}\Exp{(-1_A)\bullet(\sqrt\phi-1)}\Exp{\sqrt\phi-1}
\\
&=e^{-\int_X(\phi-1)1_Ad\mu}\Exp{(-1_A)\bullet(\phi-1)}
\\
&=e^{\mu(A)-\nu(A)}\Exp{1_{A^c}\phi-1}.
\end{aligned}
$$
Taking the mathematical expectation and using (\ref{eq:simple exp}) twice, we obtain that
$$
\mathbb{E}_{\mu^{*}}[1_{\{\omega\in X^*:\, N_{A}(\omega)=0\} }\Exp{\phi-1}]  =e^{-\nu(A)}
  =\nu^{*}(\{\omega\in X^*:\, N_{A}(\omega)=0\}).
$$
Hence, by the 
R\'enyi's theorem, 
$\nu^*\ll\mu^*$ and 
 $\frac{d\nu^{*}}{d\mu^{*}}=\Exp{\frac{d\nu}{d\mu}-1}$.
 The second claim of the theorem was proved in \cite{Tak90}.
\end{proof}

We note that  Theorem~\ref{thm:absolute continuity} highlights the connection between extended coherent vectors and Radon-Nikodym derivatives of equivalent Poisson point process measures. The following result can be seen as an explicit  description of $\frac{d\nu^{*}}{d\mu^{*}}$ for $ \nu\in\mathcal{M}_{\mu,2}^{\circ,+}$ as a function from $X^*$ to $\Bbb R$.

\begin{thm}\label{th:R-N-D}
\label{prop:LogT'}Let $\nu\in\mathcal{M}_{\mu,2}^{\circ,+}$ and set $\phi:=\frac{d\nu}{d\mu}$.
Then 
\begin{enumerate}
\item  We can represent $\log\frac{d\nu^{*}}{d\mu^{*}}$ as the following limit in probability:
\begin{multline}\label{eq:main}
\log\frac{d\nu^{*}}{d\mu^{*}}(\omega)=\lim_{\epsilon\to 0}
\bigg(\int_{\{x\in X:\,|\log\phi(x)|>\epsilon\}}
\log\phi \,d\omega  \\
-\int_{\{x\in X:\,|\log\phi(x)|>\epsilon\}}(
\phi-1)\,d\mu\bigg).
\end{multline}
\item 
Moreover, $\log\frac{d\nu^{*}}{d\mu^{*}}$ is an infinitely divisible random
variable whose L\'evy measure is the image of $\mu$ by $\log\phi$,
restricted to $\mathbb{R}\setminus\{ 0\} $.
\item
We have that
\[
\mathbb{E}_{\mu^{*}}\Big[\log\frac{d\nu^{*}}{d\mu^{*}}\Big]=-\int_{X}(\phi-1-\log \phi)\,d\mu\in[-\infty,0].
\]
It is finite if and only if 
$\int_{\{x\in X:\,|\log \phi(x)|>1\}}|\log \phi\,|d\mu<\infty$.
\end{enumerate}
\end{thm}

\begin{proof}
Given $\epsilon>0$, we let $X_\epsilon:=\{x\in X:\,|\log\phi(x)|>\epsilon\}$.
As usual, $X_\epsilon^c$ denotes the complement to $X_\epsilon$.
We first prove three auxiliary claims.

{\it Claim }A:
$\mu(X_\epsilon)<\infty$ for each $\epsilon>0$.

Indeed,
\begin{align*}
\mu(X_\epsilon) & =\mu(\{x\in X:\, \phi(x)>e^{\epsilon}\} \cup
\{x\in X:\, \phi(x)<e^{-\epsilon}\}\\
 & \le\mu(\{x\in X:\,|\sqrt{\phi(x)}-1|>\alpha\}\\
 & \le\frac{1}{\alpha^{2}}\int_{X}\left(\sqrt{\phi}-1\right)^{2}d\mu<+\infty,
\end{align*}
where $\alpha=\min\left(e^{\epsilon}-1,1-e^{-\epsilon}\right)$.

{\it Claim }B:
$\left(\log\phi\right)^{2}\wedge1\le\kappa\left(\sqrt{\phi}-1\right)^{2}$
for some constant $\kappa>0$.

This claim follows from  the standard inequality
$\log t\le t-1$ for $t>0$.

{\it Claim }C:
$\phi-1-\log\phi\cdot 1_{X_1^c}\in
L^{1}(\mu)$.

To prove this claim we first  write the function $\phi-1-\log\phi\cdot 1_{X_1^c}$
as the following sum:
$$
(\sqrt{\phi}-1)^{2}+2(\sqrt{\phi}-1-\log\sqrt{\phi})1_{X_1^c}+2(\sqrt{\phi}-1)1_{X_1}.
$$
The first term in this sum  is in $L^{1}(\mu)$ because $\sqrt\phi-1\in L^2(\mu)$.
Since
$$
0\le\sqrt{\phi}-1-\log\sqrt{\phi}\le(\sqrt{\phi}-1)^{2},
$$
it follows that the second term is in $L^1(\mu)$ too.
The Cauchy-Schwarz inequality and Claim~A yield that the third term is also integrable.
Claim~C follows.

It follows from Claim~B that the stochastic integral $I_{\mu}(\log\phi):X^*\to\Bbb R$ is well defined
(see  Appendix).
Claim~C implies  that  the real 
$$
\beta:=-\int_{X}(\phi-1-\log\phi\cdot 1_{X_1^c})\,d\mu
$$
 is well defined.
It follows that 
\begin{equation}\label{eq:stochastic__integral}
I_{\mu}(\log\phi)(\omega)+\beta=\lim_{\epsilon\to0}\bigg(\int_{X_\epsilon}\log\phi \,d\omega
-\int_{X_\epsilon}(\phi-1)\,d\mu\bigg),
\end{equation}
where the limit means the convergence in probability (see  Appendix).
Our purpose now is to identify the lefthand side of this formula as a Radon-Nikodym derivative.

It is straightforward to verify that for each subset $B\in\mathcal
 A$, we have that
$
 (\phi-1)1_{B}=((\sqrt{\phi}-1)1_{B})\bullet ((\sqrt{\phi}-1)1_{B}).
$
Since $\Exp{(\sqrt{\phi}-1)1_{X_\epsilon}}\to\Exp{\sqrt{\phi}-1}$ in $L^2(\mu^*)$,
we can apply (\ref{eq:product formula}) to obtain that
\begin{align}\label{eq:split}
\begin{split}
\Exp{(\phi-1)1_{X_\epsilon}} &=e^{-\int_{X_\epsilon}(\sqrt{\phi}-1)^{2}d\mu}\,
\Exp{(\sqrt{\phi}-1)1_{X_\epsilon}}^{2}\\
&\to e^{-\int_{X}(\sqrt{\phi}-1)^{2}d\mu}\,
\Exp{(\sqrt{\phi}-1)}^{2}\\
&=\Exp{\phi-1} 
\end{split}
\end{align}
 in $L^1(\mu^*)$ as $\epsilon\to 0$.
Since $(\sqrt{\phi}-1)1_{X_\epsilon}\in \mathcal{B}_{0}(X)$, it follows from (\ref{eq:coherent_simple function-1}) that for a.e. $\omega\in X^*$,
$$
\begin{aligned}
\Exp{(\phi-1)1_{X_\epsilon}}(\omega) & =e^{-\int_{X_\epsilon}(\sqrt{\phi}-1)^{2}d\mu
-2\int_{X_\epsilon}(\sqrt\phi-1)d\mu}
 \prod_{\{x\in X_\epsilon:\,\omega(\{x\})=1\}}\phi(x)
\\
 & =e^{\int_{X_\epsilon}\log\phi d\omega-\int_{X_\epsilon}\left(\phi-1\right)d\mu}.
\end{aligned}
$$
From this and (\ref{eq:split}) we deduce that
$$
\lim_{n\to\infty}
\left(\int_{X_\epsilon}\log\phi \,d\omega-\int_{X_\epsilon}\left(\phi-1\right)d\mu\right)
=\log \Exp{\phi-1}
$$
where the limit is in $\mu^*$-probability. This formula, (\ref{eq:stochastic__integral}) and Theorem~\ref{eq:stochastic__integral} yield that
\begin{equation}\label{eq:beta}
I_{\mu}(\log\phi)+\beta=\log\frac{d\nu^{*}}{d\mu^{*}}.
\end{equation}
Thus, (1) is proved. Moreover, $I_\mu(f)$ is infinitely divisible (see Appendix)) so (\ref{eq:beta})
 implies~(2).

Since 
$
\phi-1-\log \phi\ge0,
$
the integral $\int_{X}(\phi-1-\log \phi)\,d\mu$ is always well
defined. 
Combining this observation with Claim~C,
we obtain that
\begin{equation}\label{eq:equivalence}
\int_{X}(\phi-1-\log \phi)\,d\mu<\infty\Longleftrightarrow \int_{X_1}|\log \phi|\,d\mu<\infty.
\end{equation}
By Proposition \ref{prop: appendix ID}, the latter inequality is equivalent to the fact that $ I_\mu(\log\phi)\in
L^{1}(\mu^{*})$.
The latter, in turn, is equivalent to $\log\frac{d\nu^*}{d\mu^*}\in
L^{1}(\mu^{*})$ in view of (\ref{eq:beta}).

Firstly we consider  the case where $\log\frac{d\nu^*}{d\mu^*}\in
L^{1}(\mu^{*})$.
Then it follows from~(\ref{eq:beta}), Proposition~\ref{prop: appendix ID} and the definition of $\beta$ that  
\begin{align*}
\mathbb{E}_{\mu^*}\bigg[\log \frac{d\nu^*}{d\mu^*}\bigg]&= \mathbb{E}_{\mu^*}(I_\mu(\log\phi))+\beta\\
&= \int_{X_1}\log\phi\,d\mu+\beta\\
&=\int_X (\log\phi-\phi-1)d\mu
\end{align*}
Consider now the second case, where 
$\int_{X_1}|\log \phi|\,d\mu=+\infty$.
Then from~(\ref{eq:equivalence}) we deduce that $\int_{X}(\phi-1-\log \phi )d\mu=+\infty$. 
On the other hand, by the Jensen inequality,
$\mathbb{E}_{\mu^{*}}[-\log \frac{d\nu^*}{d\mu^*}]\ge 0$.
Hence the fact  $\log \frac{d\nu^*}{d\mu^*}\notin L^1(\mu^*)$ implies $\mathbb{E}_{\mu^{*}}[-\log \frac{d\nu^*}{d\mu^*}]=+\infty$. 
 The proof of (3) is now complete.
\end{proof}

\begin{rem}
With additional efforts, using the martingale convergence theorem, it is
possible to prove the almost sure convergence in (\ref{eq:main}) instead of the convergence
in probability.

\end{rem}

The special case where $\phi-1\in L^1(\mu)$ has been considered in \cite{Tak90}.
In this case we have the following results.

\begin{thm}\label{thm: LogT'}
	Let $\nu\in\mathcal{M}_{\mu,1}^{\circ,+}$ and set $\phi:=\frac{d\nu}{d\mu}$.
	Then
	\begin{enumerate} 
		\item $\log\phi\in L^{1}(\omega)$ for $\mu^{*}$-almost every $\omega\in X^{*}$ and
		$$
		\frac{d\nu^{*}}{d\mu^{*}}\left(\omega\right)=e^{-\int_{X}\left(\phi-1\right)d\mu}\prod_{\{x\in X:\,\omega\left(\{x\}\right)=1\}}\phi\left(x\right),
		$$
		where the infinite product converges absolutely.
		\item
	The integral $\int_{X}\log \phi \,d\mu$ is well defined and takes 
		values in the extended interval $\left[-\infty,\int_X(\phi-1)d\mu\right]$. 
		Moreover,
		\[
		\mathbb{E}_{\mu^{*}}\Big[\log\frac{d\nu^{*}}{d\mu^{*}}\Big]=-\int_X(\phi-1)\,d\mu+\int_{X}\log \phi\,d\mu.
		\]
	\end{enumerate}
\end{thm}
\begin{proof}
In the course of  proof  we will  use the notation $X_\epsilon$  and $X_1^c$ and refer to
Claims~B and C 
from the proof of Theorem~\ref{prop:LogT'}.

If $\nu\in\mathcal{M}_{\mu,1}^{\circ,+}$ then
$\phi-1\in L^{1}(\mu)$ and hence by Claim C, 
$(\log\phi)1_{X_1^c}\in L^{1}(\mu)$.
By Claim B, $|\log\phi|^21_{X_1^c}\in L^1(\mu)$, thus the stochastic integral $I_\mu(|\log\phi|)$ is well defined as 
\begin{equation}\label{eq:stoc}
I_{\mu}(|\log\phi|)=\lim_{\epsilon\to0}\bigg(\int_{X_{\epsilon}}|\log\phi|\,d\omega-\int_{X_{\epsilon}}|\log\phi|1_{X_1^c}\,d\mu\bigg),
\end{equation}
where the limit is in probability. 
In particular, $I_{\mu}(|\log\phi|)$ is finite $\mu^{*}$-almost surely.
By the monotone convergence theorem and the integrability of  $|\log\phi|1_{X_1^c}$,
$$
\lim_{\epsilon\to 0}\int_{X_{\epsilon}}|\log\phi|1_{X_1^c}\,d\mu=
\int_{X}|\log\phi|1_{X_1^c}\,d\mu<\infty.
$$
It follows from this and (\ref{eq:stoc}) that
there exists $\lim_{\epsilon\to0}\int_{X_{\epsilon}}|\log\phi|\,d\omega<\infty$ for a.e. $\omega\in X^*$.
By the monotone convergence theorem,  
$$
\lim_{\epsilon\to0}\int_{X_{\epsilon}}|\log\phi|\,d\omega=\int_{X}|\log\phi|\,d\omega.
$$
Thus,  $\log\phi\in L^{1}(\omega)$ for $\mu^{*}$-a.e.
This fact combined with the integrability of $\phi-1$ imply the almost everywhere convergence in
(\ref{eq:main}).
Passing to this limit, we obtain now that
\[
\log\frac{d\nu^{*}}{d\mu^{*}}(\omega)=\int_{X}\log\phi \,d\omega-\int_{X}(\phi-1)\,d\mu
\]
for a.e. $\omega$.
This proves  (1).
The second claim follows from Theorem \ref{prop:LogT'}(3).
\end{proof}

\begin{rem}
\begin{itemize}
\item
Firstly, we note that the formula for the Radon-Niko\-dym derivative in Theorem \ref{thm: LogT'} is well known and follows immediately from Theorem~\ref{prop:LogT'}(1) and the fact that $\phi-1\in L^1(\mu)$.
However the absolute convergence of the infinite product (or the fact that $\log \phi\in L^1(\omega)$ for a.e. $\omega$) requires an additional reasoning. 
\item
Secondly, it is worth mentioning that the combination of $\nu\in\mathcal{M}_{\mu,2}^{\circ,+}$ and $\log\phi\in L^{1}(\mu)$
implies that $\nu$ is in $\mathcal{M}_{\mu,1}^{\circ,+}$ and $\log\frac{d\nu^{*}}{d\mu^{*}}\in L^1(\mu^*)$.
Indeed, from $\nu\in\mathcal{M}_{\mu,2}^{\circ,+}$ we get that $\phi-1-(\log\phi)1_{X_{1}^{c}}\in L^{1}(\mu)$ (see Claim~C).
On the other hand, the fact $\log\phi\in L^{1}(\mu)$ implies that $(\log\phi)1_{X_{1}^{c}}\in L^{1}(\mu)$.
Therefore $\phi-1\in L^{1}(\mu)$, i.e. $\nu\in\mathcal{M}_{\mu,1}^{\circ,+}$.
The integrability of $\log\frac{d\nu^{*}}{d\mu^{*}}$ follows now from Theorem~\ref{thm: LogT'}(2).
\end{itemize}
\end{rem}

\section{Poisson suspensions of nonsingular transformations and related Koopman representations}\label{sec: NS of Poisson and Koopman}

\subsection{The unitary Koopman representation of the group  of nonsingular transformations}
Let $(Y,\mathcal{B},\rho)$ be a $\sigma$-finite Lebesgue space.
Denote by $\mathcal{U}(L^{2}(\rho))$ the group of unitary operators in $L^2(\mu)$
and by $\mathcal{U}_{\Bbb R}(L^{2}(\rho))$ the  subgroup of unitaries that preserve invariant
the $\Bbb R$-subspace $L^2_\Bbb R(\rho)$ of real valued functions in $L^2(\mu)$.
Let
Aut$(Y,\mathcal{B},\rho)$ stand for the group of all nonsingular transformations of $(Y,\mathcal{B},\rho)$.
For each 
   $S\in\text{Aut}(Y,\mathcal{B},\rho)$, we set
 $S^{\prime}:=\frac{d\rho\circ S^{-1}}{d\rho}$
and 
define a unitary operator $U_S\in \mathcal{U}_\Bbb R(L^{2}(\rho))$ by setting $U_Sf:=\sqrt{S'}f\circ  S^{-1}$.
Then
 the mapping 
 $$
 U:\text{Aut}(Y,\mathcal{B},\rho)\ni S\mapsto U_S\in \mathcal{U}_\Bbb R(L^{2}(\rho))
 $$
 is a unitary one-to-one representation of Aut$(Y,\mathcal{B},\rho)$ in $L^2(Y,\rho)$.
 It is called {\it the unitary Koopman representation} of Aut$(Y,\mathcal{B},\rho)$.
 
 Let  $C_0:=\{f\in L^2(\rho):\, f\ge 0\}$.
Then $C_0$ is a closed cone in $L^2(\rho)$.
It is well known that
 \begin{equation}\label{eq:positive unitary}
\{V\in \mathcal{U}(L^{2}(\rho)): \, VC_0=C_0\} =\{U_S:\, S\in\text{Aut}(Y,\mathcal B,\rho)\}.
 \end{equation}
 Endow $\mathcal{U}_\Bbb R(L^{2}(\rho))$ with the weak (equivalently, strong) operator topology.
We recall that {\it the weak topology} on $\text{Aut}(Y,\mathcal B,\rho)$ is the weakest topology in which $U$ is continuous.
It follows from (\ref{eq:positive unitary}) that $\text{Aut}(Y,\mathcal B,\rho)$ furnished with the weak topology is a Polish group.

\subsection{Nonsingular Poisson suspensions and related transformation groups}
Let the measure space $(X,\mathcal A,\mu)$ be as in the previous section and let $T$ be a nonsingular (invertible) transformation of this space.
Theorem~\ref{thm:absolute continuity} provides  us with an ``if and
only if'' criteria for when $T_{*}$ is non-singular and Theorems~\ref{prop:LogT'} and
Theorem~\ref{thm: LogT'} give an explicit
pointwise description of the  Radon-Nikodym derivative of $T_*$ as follows.

\begin{cor}\label{cor:R-N}
$T_{*}$ is a nonsingular automorphism of $\left(X^{*},\mathcal{A}^{*},\mu^{*}\right)$
if and only if  $\sqrt{T^{\prime}}-1\in L^{2}(\mu)$. In this
case $\left(T_{*}\right)^{\prime}=\Exp{T^{\prime}-1}$.
Moreover,
\begin{enumerate}
\item 
We can represent $\log( T_*)'$ as the following limit in probability:
\begin{multline*}
\log( T_*)'(\omega)=\lim_{\epsilon\to0}\bigg(\int_{\{x\in X:\,|\log T'(x)|>\epsilon\}}\log T'\,d\omega  \\
-\int_{\{x\in X:\,|\log T'(x)|>\epsilon\}}(T'-1)\,d\mu\bigg).
\end{multline*}
\item The function $X^*\ni\omega\mapsto \log (T_*)'(\omega)\in\Bbb R$ is an infinitely divisible random
variable whose L\'evy measure is the (restriction to $\mathbb{R}\setminus\{0\}$ of) the image of $\mu$ by $\log T'$.
\item If $T'-1\in L^1(\mu)$, then  $\log T'\in L^{1}(\omega)$ for $\mu^{*}$-almost every $\omega\in X^{*}$ and
$$
(T_*)'(\omega)=e^{-\int_{X}(T'-1)\,d\mu}\prod_{\{x\in X:\,\omega(\{x\})=1\}}T'(x),
$$
where the infinite product converges absolutely.
\end{enumerate}
\end{cor}

Our purpose in this paper is to study nonsingular Poisson suspensions.
Therefore in view of Corollary~\ref{cor:R-N} we introduce some special subgroups of nonsingular transformations that are related naturally  to these suspensions.

\begin{defn}
We set
\begin{flalign*}
\text{Aut}_{2}(X,\mathcal{A},\mu) & :=\left\{ T\in\text{Aut}(X,\mathcal{A},\mu),\sqrt{T^{\prime}}-1\in L^{2}\left(\mu\right)\right\}, \\
\text{Aut}_{1}(X,\mathcal{A},\mu) & :=\left\{ T\in\text{Aut}(X,\mathcal{A},\mu),T^{\prime}-1\in L^{1}\left(\mu\right)\right\}\quad\text{and} \\
\text{Aut}_{\mathcal{P}}(X^{*},\mathcal{A}^{*},\mu^{*}) & :=\{ T_{*}\in\text{Aut}(X^{*},\mathcal{A}^{*},\mu^{*}),T\in\text{Aut}_{2}(X,\mathcal{A},\mu)\}. 
\end{flalign*}
\end{defn}

Of course, Aut$_{1}(X,\mathcal{A},\mu)\subset\text{Aut}_{2}(X,\mathcal{A},\mu)$.
By Remark~\ref{rem: measures}(\ref{rem:equivalence}),  the two objects are
subgroups of $\text{Aut}(X,\mathcal{A},\mu)$. 
Since the map $T\mapsto T_*$ is a homomorphism from $\text{Aut}_{2}(X,\mathcal{A},\mu) $
to $\text{Aut}(X^{*},\mathcal{A}^{*},\mu^{*})$, the set
 $\text{Aut}_{\mathcal{P}}(X^{*},\mathcal{A}^{*},\mu^{*})$
is a subgroup of $\text{Aut}(X^{*},\mathcal{A}^{*},\mu^{*})$.

\begin{defn} The map $\text{Aut}_{2}(X,\mathcal{A},\mu)\ni T\mapsto T^*
\in\text{Aut}(X^{*},\mathcal{A}^{*},\mu^{*})$ will be called
{\it the Poisson homomorphism}.
\end{defn}

In the next three subsections we study Aut$_{2}(X,\mathcal{A},\mu)$, $\text{Aut}_{1}(X,\mathcal{A},\mu)$ and $\text{Aut}_{\mathcal{P}}(X^{*},\mathcal{A}^{*},\mu^{*})$ respectively
 in more detail.

\subsection{Polish group ${\rm Aut}_{2}(X,\mathcal{A},\mu)$ and the associated affine Koopman representation}
 We denote by Aff$_\Bbb R(L^2(\mu))$ the subgroup of  invertible affine operators in  $L^2(\mu)$ that preserve invariant the $\Bbb R$-subspace $L^2_\Bbb R(\mu)$.
Then Aff$_\Bbb R(L^2(\mu)):=L^2_\Bbb R(\mu)\rtimes\mathcal{U}_\Bbb R(L^{2}(\mu))$.
We recall that an operator $A=(f,V)\in\text{Aff}_\Bbb R(L^2(\mu))$ acts on $L^{2}(\mu)$ by the formula
 $Ah:=f+Vh$.
 One can verify that the multiplication law in  Aff$_\Bbb R(L^2(\mu))$ is given by:
 $$
 (f,V)(f',V'):=(f+Vf', VV').
 $$
Aff$_\Bbb R(L^2(\mu))$ is a Polish group when  endowed with the product of the  norm topology on $L^2(\mu)$ and the weak operator topology on 
$\mathcal{U}_\Bbb R(L^{2}(\mu))$.
 We now let
$$
 C_{-1}:=\{f\in L^2(\mu):\, f\ge -1\}.
 $$
 Then $C_{-1}$ is a closed semispace in $L^2_\Bbb R(\mu)$.
 We now establish an ``affine'' analogue of (\ref{eq:positive unitary}).
 
 \begin{thm}\label{th:characterization Aut_2}
  $$
\{A\in\text{{\rm Aff}}_\Bbb R(L^{2}(\mu)): \, AC_{-1}=C_{-1}\} =\{(\sqrt{ S'}-1, U_S):\, S\in\text{{\rm Aut}}_2(X,\mathcal A,\mu)\}.
 $$
 \end{thm}
 
 \begin{proof}
 If $A:=(\sqrt{ S'}-1, U_S)$ for some transformation $S\in\text{{\rm Aut}}_2(Y,\mathcal B,\rho)$ then
 $Ah=(h\circ S^{-1}+1)\sqrt{S'}-1\ge -1$ for each $h\in C_{-1}$.
 Hence $AC_{-1}\subset C_{-1}$.
 The same is true if we take $S^{-1}$ in place of $S$. 
Therefore we obtain that $A^{-1}C_{-1}\subset C_{-1}$.
Hence $AC_{-1}=C_{-1}$, as desired.
 
  Conversely, let $A=(f, V)\in \text{{\rm Aff}}_\Bbb R(L^{2}(\mu))$ and $AC_{-1}=C_{-1}$.
 The following properties are verified straightforwardly:
 \begin{itemize}
 \item $C_0+C_{-1}=C_{-1}$.
 \item If $a+C_{-1}\subset C_{-1}$ for some $a\in L^2(\mu)$ then $a\in C_0$.
 \end{itemize}
 Then $A(C_0+C_{-1})=C_{-1}$.
 On the other hand, 
 $$
 A(C_0+C_{-1})= AC_0 + AC_{-1} -A0=VC_0+C_{-1}.
 $$
 Therefore $VC_0 +C_{-1}=C_{-1}$.
 Hence $VC_0\subset C_0$.
 Since $A^{-1}C_{-1}=C_{-1}$, a similar reasoning yields that  $V^{-1}C_0\subset C_0$.
 Therefore $VC_0=C_0$.
 In view of~(\ref{eq:positive unitary}), there is $S\in \text{Aut}(X,\mathcal A,\mu)$ such that
 $V=U_S$.
 Hence 
 \begin{equation}\label{eq: C}
 VC_{-1}=\{h\circ S\sqrt{S'}+f|\, h\in C_{-1}\}=C_{-1}.
 \end{equation}
 Let $L(\mu)$ stand for the space of all measurable real valued functions on $X$.
 Endowed with the natural order, $L(\mu)$ is an ordered vector space.
 Considering the semispace $C_{-1}$ as a subset of $L(\mu)$, we deduce from (\ref{eq: C}) that
 $$
L(\mu)\ni -1= \inf C_{-1}=\inf VC_{-1}=-\sqrt{S'}+f\in L(\mu).
 $$
 Thus, $\sqrt{S'}-1=f\in L^2_\Bbb R(\mu)$.
  \end{proof}
 
 We note that Theorem~\ref{th:characterization Aut_2} provides an alternative characterization of Aut$_2(X,\mathcal A,\mu)$.
 This characterization   is not  related straightforwardly to Poisson suspensions.
 Moreover, Theorem~\ref{th:characterization Aut_2} determines a one-to-one representation of Aut$_2(X,\mathcal A,\mu)$ in Aff$_\Bbb R(L^2(\mu))$.

 \begin{defn}\label{def:affine Koopman} We call the homomorphism
 $$
 A^{(2)}: \text{{\rm Aut}}_2(X,\mathcal A,\mu)\ni S\mapsto A^{(2)}_S:=(\sqrt{S'}-1,U_S)\in
  \text{Aff}_\Bbb R(L^2(\mu))
  $$
  {\it the affine Koopman representation} of Aut$_2(X,\mathcal A,\mu)$.
  We call the weakest topology on Aut$_2(X,\mathcal A,\mu)$ in which the affine Koopman representation is continuous {\it the $d_2$-topology}.
 \end{defn}

Thus, a sequence $(T_n)_{n=1}^\infty$ of transformations $T_n\in\text{{\rm Aut}}_2(X,\mathcal A,\mu)$ converges in $d_2$ to a transformation $T\in \text{{\rm Aut}}_2(X,\mathcal A,\mu)$ as $n\to \infty$
if and only if  $T_n\to T$ weakly and $\|\sqrt{T_n'}-\sqrt{T'}\|_2\to 0$ as $n\to\infty$.
It follows from Theorem~\ref{th:characterization Aut_2} that the image of  Aut$_2(X,\mathcal A,\mu)$
under $A^{(2)}$ is closed in $\text{Aff}(L^2(\mu))$.
Hence
 Aut$_2(X,\mathcal A,\mu)$ endowed with the $d_2$-topology is a Polish group.
 We state the next proposition without proof.
 It follows easily from~Remark~\ref{rem: measures}(\ref{rem:equivalence}).

\begin{prop} Let $\nu\in\mathcal M_{\mu,2}^{\circ,+}$.
Then ${\rm Aut}_2(X,\mathcal A,\nu)=\text{{\rm Aut}}_2(X,\mathcal A,\mu)$  as topological groups furnished with the corresponding $d_2$-topologies.
\end{prop}

\subsection{Polish group ${\rm Aut}_{1}(X,\mathcal{A},\mu)$, the associated affine Koopman representation and the structure of semidirect product}
Let $\mathcal{U}(L^{1}(\mu))$ stand for the group of isometries in $L^1(\mu)$ and
let  $\mathcal{U}_\Bbb R(L^{1}(\mu))$ stand for the subgroup of isometries that preserve invariant the $\Bbb R$-subspace  $L^1_\Bbb R(\mu)$ of real valued functions in $L^1(\mu)$.
We denote  by  Aff$_\Bbb R(L^1(\mu)):=L^1_\Bbb R(\mu)\rtimes\mathcal{U}_\Bbb R(L^{1}(\mu))$ the group of invertible affine operators
 in $L^{1}(\rho)$ that preserve invariant $L^1_\Bbb R(\mu)$.
The multiplication law in Aff$_\Bbb R(L^1(\mu))$ is given by the same formula as the multiplication law
in Aff$_\Bbb R(L^2(\mu))$.
We also note that Aff$_\Bbb R(L^1(\mu))$ is a Polish group when  endowed with the product of the  norm topology on $L^1_\Bbb R(\mu)$ and the strong (not the weak!)  operator topology on 
$\mathcal{U}_\Bbb R(L^{1}(\mu))$.
 We now let
$$
 C_{-1}^{(1)}:=\{f\in L^1(\mu):\, f\ge -1\}.
 $$
 Then $C_{-1}^{(1)}$ is a closed semispace in $L^1_\Bbb R(\mu)$.
Given $S\in \text{Aut}_1(X,\mathcal A,\mu)$, we define an  isometric invertible operator $U_S^{(1)}$
on $L^1(\mu)$
by setting
$U_S^{(1)}f:=f\circ S^{-1} \cdot S'$.
Then we call the one-to-one homomorphism 
$$
U^{(1)}: \text{Aut}_1(X,\mathcal A,\mu)\ni S\mapsto U_S^{(1)}\in \mathcal U_\Bbb R(L^1(\mu))
 $$
{\it the isometric Koopman representation} of  $\text{Aut}_1(X,\mathcal A,\mu)$.
 The following theorem is an  analogue of  Theorem~\ref{th:characterization Aut_2}.

 \begin{thm}\label{th:characterization Aut_1}
  $$
\{A\in\text{{\rm Aff}}_\Bbb R(L^{1}(\mu)): \, AC^{(1)}_{-1}=C_{-1}^{(1)}\} =\{( S'-1, U_S^{(1)}):\, S\in\text{{\rm Aut}}_1(X,\mathcal A,\mu)\}.
 $$
 \end{thm}

We do not provide a proof of this theorem because it is very similar to the proof of 
Theorem~\ref{th:characterization Aut_2}.

\begin{defn}\label{def:affine 1-Koopman} We call the one-to-one homomorphism
 $$
 A^{(1)}: \text{{\rm Aut}}_1(X,\mathcal A,\mu)\ni S\mapsto A_S^{(1)}:=({S'}-1,U_S^{(1)})\in
  \text{Aff}_\Bbb R(L^1(\mu))
  $$
  {\it the affine Koopman representation} of Aut$_1(X,\mathcal A,\mu)$.
  We call the weakest topology on Aut$_1(X,\mathcal A,\mu)$ in which the affine Koopman representation is continuous {\it the $d_1$-topology}.
 \end{defn}
 
 Thus, a sequence $(T_n)_{n=1}^\infty$ of transformations $T_n\in\text{{\rm Aut}}_1(X,\mathcal A,\mu)$ converges in $d_1$ to a transformation $T\in \text{{\rm Aut}}_1(X,\mathcal A,\mu)$ as $n\to \infty$
if and only if  $T_n\to T$ weakly and $\|{T_n'}-{T'}\|_1\to 0$ as $n\to\infty$.
It follows from Theorem~\ref{th:characterization Aut_1} that the image of  Aut$_1(X,\mathcal A,\mu)$
under $A^{(1)}$ is closed in $\text{Aff}_\Bbb R(L^1(\mu))$.
Hence
 Aut$_1(X,\mathcal A,\mu)$ endowed with the $d_1$-topology is a Polish group.
  We state the next proposition without proof.
 It follows easily from~Remark~\ref{rem: measures}(\ref{rem:equivalence}).
 
\begin{prop} Let $\nu\in\mathcal M_{\mu,1}^{\circ,+}$.
Then ${\rm Aut}_1(X,\mathcal A,\nu)=\text{{\rm Aut}}_1(X,\mathcal A,\mu)$  as topological groups furnished with the corresponding $d_1$-topologies.
\end{prop}

 The following important group homomorphism 
was introduced in \cite{Ner1996CatInfGroups}:

$$
\chi:\text{Aut}_{1}(X,\mathcal{A},\mu) \ni T\mapsto  \chi(T):=  \int_{X}(T'-1)\,d\mu
\in\Bbb R.
$$
Of course, 
$\chi$ depends on $\mu$.
However, 
we now show that $\chi$ does not depend on the choice of measure within the class $\mathcal{M}_{\mu,1}^{\circ,+}$.

\begin{prop} Let $\nu\in\mathcal{M}_{\mu,1}^{\circ,+}$.
Then $\chi(T)=\int_{X}(\frac{d\nu\circ T^{-1}}{d\nu}-1)\,d\nu$ for each $T\in\text{\rm Aut}_1(X,\mathcal A,\mu)$.
\end{prop}

\begin{proof}
By Remark~\ref{rem: measures}, $\mu$ and $\nu$ share the same family of subsets of finite measure.
Let  $\{ A_{n}\} _{n\in\mathbb{N}}$ be an
increasing sequence of sets of finite measure such that $\bigcup_{n=1}^\infty A_n=X$.
Then
\[
b:=\int_{X}\bigg(\frac{d\nu}{d\mu}-1\bigg)d\mu=\lim_{n\to+\infty}\int_{A_n}\bigg(\frac{d\nu}{d\mu}-1\bigg)d\mu=\lim_{n\to+\infty}(\nu(A_n)-\mu(A_n)).
\]
In a similar way, $b=\lim_{n\to+\infty}(\nu(T^{-1}A_n)-\mu(T^{-1}A_n))$.
Hence 
\begin{align*}
0&=\lim_{n\to+\infty}(\nu(A_n)-\mu(A_n)-\nu(T^{-1}A_n)+\mu(T^{-1}A_n))\\
&=\lim_{n\to+\infty}(\mu(T^{-1}A_n)-\mu(A_n))-\lim_{n\to+\infty}(\nu(T^{-1}A_n)-\nu(A_n))\\
&=\chi(T)-\int_{X}\left(\frac{d\nu\circ T^{-1}}{d\nu}-1\right)d\nu.
\end{align*}
\end{proof}

\begin{thm}\label{thm: Aut1 semidirect prod}
The homomorphism $\chi$ is $d_{1}$-continuous and the quotient group $\text{{\rm Aut}}_{1}(X,\mathcal{A},\mu)/\text{{\rm Ker}}\,\chi$ is isomorphic to $\mathbb{R}$.
In other words,
 the following short sequence
of Polish groups is exact
\[
\left\{ 1\right\} \to\text{{\rm Ker\,}}\chi\to\text{{\rm Aut}}_{1}(X,\mathcal{A},\mu)\to\mathbb{R}\to\left\{ 0\right\} 
\]
Moreover, this sequence splits, i.e. there is a continuous one-to-one homomorphism
$\sigma:\mathbb{R}\to\text{{\rm Aut}}_{1}(X,\mathcal{A},\mu)$
such that $\chi\circ\sigma=id_{\mathbb{R}}$.
\end{thm}

\begin{proof}
Of course, $\chi$ is continuous.
Next, 
there is no loss in generality if we take $\left(X,\mathcal{A},\mu\right)=\left(\mathbb{R},\mathcal{B},m\right)$
where $m$ is defined by $\frac{dm}{dx}=1_{\mathbb{R}_{-}^{*}}+2\times1_{\mathbb{R}_{+}}$.
For $t\in\Bbb R$, denote by $T_t:\Bbb R\to\Bbb R$ the translation by $t$.
Then it is easy to verify that $T_t\in \text{Aut}_1(X,\mathcal A,\mu)$ and $\chi(T_{-t})=t$.
Of course, the homomorphism $\sigma:\Bbb R\ni t\mapsto T_{-t}\in  \text{Aut}_1(X,\mathcal A,\mu)$
is continuous. 
\end{proof}

It follows from the second claim of the theorem that there is a topological
isomorphism $\theta:\text{Aut}_{1}(X,\mathcal{A},\mu)\to\text{Ker}\,\chi\rtimes_{\sigma}\mathbb{R}$
such that the following diagram commutes:
$$
\begin{CD}
\{1\} @>>> \text{Ker}\,\chi @>i>>\text{Aut}_{1}(X,\mathcal{A},\mu)  @>\chi>> \Bbb R @>>>\{0\}\\
@.    @V\text{id}VV @V\theta VV   @V\text{id}VV @.\\
\{1\} @>>> \text{Ker}\,\chi @>i>>\text{Ker}\,\chi\rtimes_{\sigma}\mathbb{R} @>\chi>> \Bbb R @>>>\{0\}
\end{CD}
$$
Thus, we have showed that $\text{Aut}_{1}(X,\mathcal{A},\mu)$ has a natural structure of semidirect
product of Ker\,$\chi$ and $\Bbb R$.

In order to state one more property of $\chi$ we need to recall a definition of conservativeness.

\begin{defn} A transformation $T\in\text{Aut}(X,\mathcal A,\mu)$ is called {\it conservative} if for each subset $A\in\mathcal A$ of positive measure, there is $n>0$ such that $\mu(T^{-1}A\cap A)>0$.
\end{defn}

We recall that a nonsingular transformation $T$ is conservative if and only if  $\sum_{k=0}^\infty U^{(1)}_{T^k}f=+\infty$ 
a.e. for each measurable function $f>0$ \cite{Aar97InfErg}.

\begin{prop}
\label{prop:Conservativity Aut1} Let $T\in\text{{\rm Aut}}_{1}(X,\mathcal{A},\mu)$.
If\, $T$ is conservative 
then $\chi(T)=0$.
\end{prop}

\begin{proof}
Let $\phi:=T'-1$.
Then $\phi\in L^1(\mu)$ and $\int_X\phi\,d\mu=\chi(T)$.
Take $g\in L^{1}(\mu)$ such that $g>0$ and $\int_{X}gd\mu=1$.
By the Hurewicz ratio ergodic theorem, there exists $\psi\in L^{1}(\mu)$
such that $\psi\circ T=\psi$, $\int_{X}\psi \,d\mu=\int_{X}\phi\,d\mu=\chi\left(T\right)$
and
\[
\frac{\sum_{k=0}^{n}U_{T^{k}}^{(1)}\phi}{\sum_{k=0}^{n}U_{T^{k}}^{\left(1\right)}g}\to\psi\quad\text{almost everywhere as $n\to\infty$.}
\]
However 
$\sum_{k=0}^nU_{T^k}^{(1)}\phi=(T^{n+1})'-1  \ge-1$
while  $\sum_{k=0}^{n}U_{T^{k}}^{(1)}g\to+\infty$
as $n\to\infty$.
Consequently, $\psi\ge0$ a.e. and hence
 $\chi\left(T\right)\ge0$.
However the transformation $T^{-1}$ is also conservative and the above reasoning yields that
$\chi(T^{-1})\ge0$. 
As $\chi(T^{-1})=-\chi(T)$,
we obtain that $\chi\left(T\right)=0$.
\end{proof}

We conclude this subsection with a discussion about relationship among $d_1$, $d_2$ and the weak topology.
Since $\|\sqrt{T'}-1\|_2^2\le \|{T'}-1\|_1$, it follows that $d_1$ is stronger than $d_2$.

\begin{prop}  \ 
\begin{itemize} 
\item
$d_1$ is strictly stronger than $d_2$ restricted to $\text{{\rm Aut}}_{1}(X,\mathcal{A},\mu)$ and
$d_2$ is strictly stronger than  the weak topology restricted to $\text{{\rm Aut}}_{2}(X,\mathcal{A},\mu)$.
\item $\text{{\rm Aut}}_{1}(X,\mathcal{A},\mu)$ is   dense and meager in $\text{{\rm Aut}}_{2}(X,\mathcal{A},\mu)$ endowed with  $d_2$ and 
$\text{{\rm Aut}}_{2}(X,\mathcal{A},\mu)$ is   dense and meager in $\text{{\rm Aut}}(X,\mathcal{A},\mu)$ endowed with the weak topology.
\item $\chi$ is not $d_2$-continuous. Hence $\chi$ does not extend by continuity to
$\text{{\rm Aut}}_{2}(X,\mathcal{A},\mu)$.
\item If $\mathcal B_1$, $\mathcal B_2$ and $\mathcal B$ stand for the Borel $\sigma$-algebras generated by $\tau_1$, $\tau_2$ and the weak topology respectively then
$\mathcal B_2\restriction \text{{\rm Aut}}_{1}(X,\mathcal{A},\mu)=\mathcal B_1$ and
$\mathcal B\restriction \text{{\rm Aut}}_2(X,\mathcal{A},\mu)=\mathcal B_2$.
\end{itemize}
\end{prop}
 
We do not provide a proof of this proposition because it will not be used below in the paper.
We only note that it can be deduced from the general theorems of the descriptive topology combined with several facts that are proved in the next section: Propositions~\ref{prop: local are dense in Aut2}, \ref{prop: Rokhlin property}, Theorem~\ref{thm: 2nd main category thm}.
The interested reader can also prove it independently of the next section by constructing appropriate concrete counterexamples.

\subsection{Unitary Koopman representations of $\text{{\rm Aut}}_{\mathcal{P}}(X^{*},\mathcal{A}^{*},\mu^{*})$}

Our objective in this subsection is to clarify relationship between  the  Koopman operators associated to $T$ and $T_{*}$ respectively.

Given an affine operator $A=(f,V)\in\text{Aff}_\Bbb R(L^2(\mu))$, we  define an  operator $W_A$ in
$L^2(\mu^*)$ by setting
\begin{equation}\label{eq:Weil}
W_A\Exp h:=e^{-\frac{\|f\|^2_2}2-\langle f,Vh\rangle_{L^2(\mu)}}\Exp{Ah}\end{equation}
for all $h\in L^2(\mu)$ and then extending $W_A$ by linearity and continuity to the entire 
$L^2(\mu^*)$.
It is shown 
in
\cite[\S\,2.2]{Gui72Sym}
that $W_A$ is well defined and $W_A\in \mathcal U(L^2(\mu^*))$.
It is called  a {\it Weyl operator}.
We observe that $W_A\in \mathcal U_{\Bbb R}(L^2(\mu^*))$ and $W_AW_B=W_{AB}$ for all
$A,B\in \text{Aff}_\Bbb R(L^2(\mu))$.
It is possible to define $W_A$ for arbitrary  $V\in \mathcal U(L^2(\mu))$ and $f\in L^2(\mu)$  by the same formula
(\ref{eq:Weil}).
Then $W_A \in \mathcal U_{\Bbb R}(L^2(\mu^*))$ if and only if  $A\in \text{Aff}_\Bbb R(L^2(\mu))$.
We leave the proof of this fact as an exercise for the reader.
We also need one more auxiliary result that follows easily from \cite[Theorem~2.1]{Gui72Sym}. 

\begin{theorem*} The map
$
W:\text{{\rm Aff}}_\Bbb R(L^2(\mu))\ni A\mapsto W_A\in \mathcal U_\Bbb R(L^2(\mu^*))
$
is a continuous  one-to-one group homomorphism.
Its image $\mathcal W:=\{W_A:\, A\in \text{{\rm Aff}}_\Bbb R(L^2(\mu))\}$ is a Polish subgroup in the induced topology.
\end{theorem*}

We will call $W$ {\it the Weyl homomorphism}.
It follows from the above auxiliary theorem and  \cite[Proposition 1.2.1]{BeKe96desc}  that $\mathcal W$ is closed
in $\mathcal U_\Bbb R(L^2(\mu^*))$ endowed with the weak operator topology.

It is well known that that if a transformation $T\in\text{Aut}_2(X,\mathcal A,\mu)$ preserves $\mu$
then the associated unitary Koopman operator $U_{T_*}$  can be written as $U_{T_*}=W_{(0,U_T)}=W_{A_T^{(2)}}$.
We now extend this result to arbitrary (nonsingular) elements of $\text{Aut}_2(X,\mathcal A,\mu)$.

\begin{thm}\label{main of 3}
Let $T\in\text{{\rm Aut}}_2(X,\mathcal A,\mu)$.
Under the natural identification of $F(L^2(\mu))$ with $L^2(\mu^*)$ described in \S\,{\rm \ref{Fock}},
we obtain  that 
$U_{T_{*}}=W_{A_T^{(2)}}$.
In other words,  the composition of the Poisson homomorphism of $\text{{\rm Aut}}_2(X,\mathcal A,\mu)$
with
 the unitary Koopman representation of {\rm Aut}$_\mathcal P(X^*,\mathcal A^*,\mu^*)$ 
equals the composition of 
the affine Koopman representation
of $\text{{\rm Aut}}_2(X,\mathcal A,\mu)$ with
the Weyl homomorphism.
\end{thm}

\begin{proof}
It is sufficient to verify that
$U_{T_{*}}\Exp f=W_{A_T^{(2)}}\Exp f$
for every simple (i.e. finite valued) function $f$
  from $\mathcal{B}_{0}(X)$.
  We  note that $f\circ T^{-1}$ is also a simple function from $\mathcal{B}_{0}(X)$
  and
\begin{align*}
  \Exp {f\circ T^{-1}}(\omega)&=e^{-\int_Xf\circ T^{-1}d\mu}\prod_{\{x\in X:\,\omega(\{x\})=1\}}(1+f(T^{-1}x))\\
  &=e^{\int_X(f-f\circ T^{-1})d\mu}\,\Exp f(T_*^{-1}\omega)
\end{align*}
  at a.e. $\omega\in X^*$.
Using this and Corollary~\ref{cor:R-N} we obtain that
   \begin{align}\label{pu}
   \begin{split}
  U_{T_{*}}\Exp f&=\sqrt{(T_{*})'}\,\Exp f\circ T_{*}^{-1}\\
 &= \sqrt{\Exp{T^{\prime}-1}}e^{\int_X(f\circ T^{-1}-f)d\mu}\,\Exp {f\circ T^{-1}}.
 \end{split}
  \end{align}
 By
the product formula (\ref{eq:product formula}) (see also (\ref{eq:phi -1})), 
$$
  \Exp{T^{\prime}-1}=e^{-\| \sqrt{T^{\prime}}-1\| _{2}^{2}}\Exp{\sqrt{T^{\prime}}-1}^{2}.
$$
Due to Lemma~\ref{lem:non-negative-coherent}, 
$
\Exp{\sqrt{T^{\prime}}-1}\ge0$ and hence 
    \begin{equation}\label{was}
  \sqrt{\Exp{T^{\prime}-1}}=e^{-\frac12\|\sqrt{T'}-1\|_2^2}\Exp{\sqrt{T^{\prime}}-1}.
\end{equation}
By a straightforward computation, 
$
(\sqrt{T'}-1)\bullet( f\circ T^{-1})=A^{(2)}_Tf
$.
  Hence, in view of  (\ref{eq:product formula}), 
  \begin{equation}\label{uha}
  \Exp{\sqrt{T'}-1}\Exp{f\circ T^{-1}}=e^{\int_X(\sqrt{T'}-1) f\circ T^{-1}d\mu}\Exp{A^{(2)}_Tf}.
  \end{equation}
  Substituting first (\ref{was}) and then (\ref{uha}) into (\ref{pu}), we obtain that
\begin{align*}
U_{T_{*}}\Exp f
 & =e^{-\frac{1}{2}\| \sqrt{T^{\prime}}-1\| _{2}^{2}+\int_{X}(-f+U_Tf)d\mu}\Exp{A_T^{(2)}f}\\
 &=e^{-\frac{1}{2}\| \sqrt{T^{\prime}}-1\| _{2}^{2}-\langle \sqrt{T^{\prime}}-1,U_Tf\rangle_{L^2(\mu)}}\Exp{A_T^{(2)}f}\\
 &=W_{A_T^{(2)}}\Exp f.
\end{align*}
\end{proof}

Thus,  the group $\mathcal W$ contains the unitary Koopman operator generated by every transformation from {\rm Aut}$_\mathcal P(X^*,\mathcal A^*,\mu^*)$.
We now show that it does not contain any Koopman operator generated by transformations from the set theoretical difference 
{\rm Aut}$(X^*,\mathcal A^*,\mu^*)\setminus \text{{\rm Aut}}_\mathcal P(X^*,\mathcal A^*,\mu^*)$.

\begin{prop} 
$
\mathcal W\cap \{U_S:\,S\in\text{{\rm Aut}}(X^*,\mathcal A^*,\mu^*)\}
=\{W_{A^{(2)}_T}:\,T\in\text{{\rm Aut}}_2(X,\mathcal A,\mu)\}.
$
\end{prop}

\begin{proof} 
Suppose that for some operator $A\in\text{{\rm Aff}}_\Bbb R(L^2(\mu))$, the unitary $W_A$  is the Koopman operator generated by a nonsingular transformation of 
$(X^*,\mathcal A^*,\mu^*)$.
Then, according to (\ref{eq:positive unitary}), $W_A$ preserves invariant the cone
$L^2_+(\mu^*)$  of non-negative functions in $L^2(\mu^*)$.
It follows from Lemma~\ref{lem:non-negative-coherent} that 
$$
L^2_+(\mu^*)\cap\{\Exp h:\, h\in L^2_\Bbb R(X,\mu)\}=
\{\Exp h:\, h\in C_{-1}\}.
$$
Hence $W_A(\{\Exp h:\, h\in C_{-1}\})\subset L^2_+(\mu^*)$.
In view of (\ref{eq:Weil}) and Lemma~\ref{lem:non-negative-coherent} this is equivalent to $AC_{-1}\subset C_{-1}$.
Since $A^{-1}$ is also a Koopman operator, a similar reasoning yields that 
$A^{-1}C_{-1}\subset C_{-1}$.
Hence $AC_{-1}=C_{-1}$.
By Theorem~\ref{th:characterization Aut_2}, $A=A^{(2)}_T$ for some $T\in\text{{\rm Aut}}_2(X,\mathcal A,\mu)$.
Thus, we showed that
$
\mathcal W\cap \{U_S:\,S\in\text{{\rm Aut}}(X^*,\mathcal A^*,\mu^*)\}
\subset\{W_{A^{(2)}_T}:\,T\in\text{{\rm Aut}}_2(X,\mathcal A,\mu)\}.
$
The converse inclusion was  established in Theorem~\ref{main of 3}.
\end{proof} 

Since $\mathcal W$ is a closed subgroup of $\mathcal U(L^2(\mu^*))$, we obtain the following corollary from the above proposition.

\begin{cor} {\rm Aut}$_\mathcal P(X^*,\mathcal A^*,\mu^*)$ is  weakly closed  in 
{\rm Aut}$(X^*,\mathcal A^*,\mu^*)$.
\end{cor}

\section{Generic  properties in $\text{Aut}_{2}(X,\mathcal{A},\mu)$ and $\text{Aut}_{1}(X,\mathcal{A},\mu)$}\label{sec: Baire category}
As the groups $\text{Aut}_{2}(X,\mathcal{A},\mu)$ and $\text{Aut}_{1}(X,\mathcal{A},\mu)$ are Polish, it is natural to ask: which dynamical  properties (or, more rigorously,  the subsets of elements possessing these properties) are generic in these groups in the Baire category sense?
 Recall that a set in a Polish space is generic if it contains a subset which is a dense
  $G_\delta$  in this space.
 
 We first list the well known generic properties for  $\text{Aut}(X,\mathcal{A},\mu)$,
 the definitions of these (and other) properties will be given just below  the Theorem.
\begin{theorem*}[\cite{ChKa79}, \cite{ChHaPr87}]
The following  subsets of nonsingular transformations: 
\begin{itemize}
	\item $\mathrm{Cons}(X,\mathcal{A},\mu):=\{T\in \mathrm{Aut}(X,\mathcal{A},\mu): T\text{ is conservative}\}$,
	\item $\mathrm{Erg}(X,\mathcal{A},\mu):=\{T\in \mathrm{Aut}(X,\mathcal{A},\mu):\text{ $T$ is ergodic}\}$,
	\item $\mathrm{Erg}^{\mathrm{III}_1}(X,\mathcal{A},\mu):=\{T\in \mathrm{Aut}(X,\mathcal{A},\mu):\text{ $T$ is of type $\mathrm{III}_1$}\}$ 
\end{itemize}
are dense $G_\delta$ in {\rm Aut}$(X,\mathcal{A},\mu)$ endowed with  the weak topology.
\end{theorem*} 
Recall that a transformation $T\in \text{Aut}\left(X,\mathcal{A},\mu\right)$ is \textit{conservative} if every wandering set $W\in\mathcal{A}$ for $T$, that is a set such that $\left\{T^nW\right\}_{n\in\mathbb{Z}}$ are pairwise disjoint, is a null set. The transformation $T$ is \textit{ergodic} if every   $T$-invariant subset is either $\emptyset$ or $X$ modulo null sets.
The transformation $T$ is \textit{aperiodic} if there is a conull subset $X'\subset X$ such that for all $x\in X'$ and $n>0$, $T^nx\neq x$. There are several ways to define the type $\mathrm{III}_1$ property, in this paper we will use the definition involving the Maharam extension. 
The \textit{Maharam extension} of $T\in \text{Aut}(X,\mathcal{A},\mu)$, is the transformation $\tilde{T}$ on $\left(X\times \mathbb{R}, \mathcal{A}\otimes \mathcal{B}_{\mathbb{R}}\right)$ defined by 
\[
\tilde{T}(x,y):=\left(Tx,y-\log\frac{d\mu\circ T}{d\mu}(x)\right).
\]
For every $T\in \text{Aut}(X,\mathcal{A},\mu)$, its Maharam extension preserves the measure $\tilde{\mu}$ defined by the formula
$\tilde{\mu}(A\times I):=\mu(A)\int_I e^tdt$
for every $A\in\mathcal{A}$ and each interval $I\subset \mathbb{R}$.
  The transformation $T$ is \textit{of type} $\mathrm{III}_1$ if its Maharam extension is ergodic. 
  
The proof of genericity results usually consists of two steps: the first one is to show that the set under consideration is $G_\delta$ which may involve making use of ergodic theorems and countably many conditions defining the set. The second step is to prove that this set is dense. For the second step the following conjugacy lemma is often a key ingredient. 

\begin{lem*}
For every aperiodic $T\in \text{{\rm Aut}}(X,\mathcal{A},\mu)$, the conjugacy class 
$$
\left\{STS^{-1}:\ S\in \text{{\rm Aut}}\left(X,\mathcal{A},\mu\right)\right\}
$$
 of $S$ is weakly dense in $\text{{\rm Aut}}(X,\mathcal{A},\mu)$  \cite{ChKa79}. 
\end{lem*}

We also recall that a Polish group $G$ has the \textit{Rokhlin property} if it has a dense conjugacy class.
For instance, by the above lemma, $\text{Aut}(X,\mathcal{A},\mu)$ has the Rokhlin property.

We note that $\text{Aut}_1(X,\mathcal{A},\mu)$  appeared in \cite{Ner1996CatInfGroups}  as a generalization of earlier work on representation theory of groups of diffeomorphisms on non-compact manifolds which are the identity outside a compact set (see \cite{GelGraVer}, \cite{GolGroPowSha}).
In this connection, it seems natural to introduce the following definition:
 a transformation $T\in \text{Aut}(X,\mathcal{A},\mu)$  is \textit{local} if there is a set $A\in\mathcal{A}$ with $\mu(A)<\infty$ such that $Tx=x$ for all $x\notin A$. 
 Denote by $\text{Aut}_{\widehat{0}}(X,\mathcal{A},\mu)\subset \text{Aut}_1(X,\mathcal{A},\mu)$ the set of all local transformations.
 Of course, $\text{Aut}_{\widehat{0}}(X,\mathcal{A},\mu)$ is a group.
 Since each transformation in $\text{Aut}_2(X,\mathcal{A},\mu)$ preserves the class
 of subsets of finite measure, $\text{Aut}_{\widehat{0}}(X,\mathcal{A},\mu)$
 is a normal subgroup of 
 $\text{Aut}_{1}(X,\mathcal{A},\mu)$ and $\text{Aut}_{2}(X,\mathcal{A},\mu)$.
  (However it is not normal in  $\text{Aut}(X,\mathcal{A},\mu)$.)
  The following consequence of Theorem \ref{thm: Aut1 semidirect prod} (specifically the topological isomorphism to a semidirect product) shows that   $\text{Aut}_{\widehat{0}}\left(X,\mathcal{A},\mu\right)$ is ``too small'' in $\text{Aut}_1(X,\mathcal{A},\mu)$. 
  
\begin{cor}
The group $\text{{\rm Aut}}_1(X,\mathcal{A},\mu)$ does not have the Rokhlin property.
The subgroup  of local transformations, being a subset of $\ker \chi$, is  $d_1$-nowhere dense  in $\text{{\rm Aut}}_1(X,\mathcal{A},\mu)$.  
\end{cor}

The situation with $\text{Aut}_2(X,\mathcal{A},\mu)$ is different. 
We preface the statement of the corresponding result with the following notation that will be of wide use in this section. 
Given $A,B\subset X$ which are both of finite measure, let $\tau_{A,B}$ denote a $\mu$-nonsingular bijection from $A$ to $B$ such that for all $x\in A$, $\frac{d\mu\circ\tau_{A,B}}{d\mu}(x)=\frac{\mu(B)}{\mu(A)}$.

\begin{prop}\label{prop: local are dense in Aut2}
$\text{{\rm Aut}}_{\widehat{0}}(X,\mathcal{A},\mu)$ is $d_2$-dense in $\text{{\rm Aut}}_2(X,\mathcal{A},\mu)$. 
\end{prop}
 
\begin{proof}
Let $T\in \text{Aut}_2(X,\mathcal{A},\mu)$. There exists an increasing sequence $A_n\in\mathcal{A}$ of finite measure subsets satisfying $\bigcup_{n=1}^\infty A_n=X$ and 
\[
\int_{X\setminus A_n}\left(1-\sqrt{\frac{d\mu\circ T^{-1}}{d\mu}}\,\right)^2d\mu=:\epsilon_n\xrightarrow[]{n\to\infty}0.
\] 
Select subsets $B_n\subset X$ of finite measure such that $A_n\cup TA_n\subset B_n$ and 
\[
\sqrt{\mu\left(B_n\setminus A_n\right)}+\sqrt{\mu\left(B_n\setminus T^{-1}A_n\right)}\geq n\left(\mu\left(A_n\right)+\mu\left(T^{-1}A_n\right)\right). 
\]
We define $T_n\in\text{Aut}_{\widehat 0}(X,\mathcal A,\mu)$ via
$$
T_n^{-1}x:=\begin{cases}
T^{-1}x, & x\in A_n\\
\tau_{B_n\setminus A_n,B_n\setminus T^{-1}A_n}x, & x\in B_n\setminus A_n\\
x, & x\notin B_n.
\end{cases}
$$
We will now show that $T_n$ converges in $d_2$ to $T$. Since $T_n^{-1}x=T^{-1}x$ for $x\in A_n$ we see that $T_n$ converges weakly to $T$ as $n\to\infty$. 
As 
\begin{equation}\label{RN for T_n}
\frac{d\mu\circ T_n^{-1}}{d\mu}(x):=
\begin{cases}
\frac{d\mu\circ T^{-1}}{d\mu}(x), & x\in A_n\\[0.5em]
\frac{\mu\left(B_n\setminus T^{-1}A_n\right)}{\mu\left(B_n\setminus A_n\right)}, & x\in B_n\setminus A_n\\
1, & x\notin B_n,
\end{cases}
\end{equation}
we deduce that
\begin{align*}
\left\Vert \sqrt{\frac{d\mu\circ T^{-1}}{d\mu}}-\sqrt{\frac{d\mu\circ T_n^{-1}}{d\mu}}\right\Vert _2 &
\le \left\Vert \left(\sqrt{\frac{d\mu\circ T^{-1}}{d\mu}}-1\right)1_{X\setminus A_n}\right\Vert _2
+\left\Vert \left(\sqrt{\frac{d\mu\circ T_n^{-1}}{d\mu}}-1\right)1_{X\setminus A_n}\right\Vert _2\\
&\leq \sqrt{\epsilon_n}+ \left(\int_{B_n\setminus A_n}\left(\sqrt{\frac{\mu\left(B_n\setminus T^{-1}A_n\right)}{\mu\left(B_n\setminus A_n\right)}}-1\right)^2d\mu\right)^{1/2}\\
&=\sqrt{\epsilon_n}+\left|\frac{\mu\left(B_n\setminus A_n\right)-\mu\left(B_n\setminus T^{-1}A_n\right)}{\sqrt{\mu\left(B_n\setminus A_n\right)}+\sqrt{\mu\left(B_n\setminus T^{-1}A_n\right)}}\right|\\
&\leq  \sqrt{\epsilon_n}+\frac{\mu\left(A_n\right)+\mu\left(T^{-1}A_n\right)}{\sqrt{\mu\left(B_n\setminus A_n\right)}+\sqrt{\mu\left(B_n\setminus T^{-1}A_n\right)}}\\
&\leq \sqrt{\epsilon_n}+\frac{1}{n}.
\end{align*}
It follows that $\lim_{n\to\infty}\left\Vert \sqrt{\frac{d\mu\circ T^{-1}}{d\mu}}-\sqrt{\frac{d\mu\circ T_n^{-1}}{d\mu}}\right\Vert _2=0$.
Thus we have shown that $T_n$ converges to $T$   in  $d_2$ as $n\to\infty$. 
\end{proof}

For $A\in\mathcal A$,  we can consider  the group $\text{Aut}_2(A,\mathcal{A}\cap A,\mu|_A)$ as the subset of $S\in \text{Aut}_2(X,\mathcal A,\mu)$ such that $S^{-1}A=A$ and $S|_{X\setminus A}=\text{id}|_{X\setminus A}$. 
We note  that if $\mu(A)<\infty$ then the map $S\mapsto S|_A$ is a topological isomorphism of  $\text{Aut}_2(A,\mathcal{A}\cap A,\mu|_A)$ with the $d_2$-topology onto $\text{Aut}(A,\mathcal{A}\cap A,\mu|_A)$ with the weak topology.

\begin{cor}\label{corol}\ 
\begin{enumerate}
\item
The subset of periodic transformations from $\text{{\rm Aut}}_{\widehat{0}}(X,\mathcal{A},\mu)$ is $d_2$-dense in $\text{{\rm Aut}}_{2}(X,\mathcal{A},\mu)$.
\item
The subset of conservative transformations is a dense $G_\delta$-subset  of $\text{{\rm Aut}}_{2}(X,\mathcal{A},\mu)$.
\end{enumerate}
\end{cor}
\begin{proof} 
(1) By \cite{IonTul65} and the metric isomorphism mentioned above, for every subset $A\in\mathcal{A}$ of finite measure, the subset of periodic transformations in 
$\text{Aut}_2(A,\mathcal{A}\cap A,\mu|_A)$ is $d_2$-dense in $\text{Aut}_2(A,\mathcal{A}\cap A,\mu|_A)$. 
This implies that the subset of periodic transformations  is $d_2$-dense in $\text{{\rm Aut}}_{\widehat{0}}(X,\mathcal{A},\mu)$. 
The claim~(1) follows from this and Proposition \ref{prop: local are dense in Aut2}.

(2)
Since  $d_2$  is stronger than the weak topology and 
the subset of conservative transformations is  a  $G_\delta$  in $\text{Aut}(X,\mathcal{A},\mu)$, it follows that
 the subset of conservative transformations  is a $G_\delta$  in 
 $\text{{\rm Aut}}_{2}(X,\mathcal{A},\mu)$.
 As every periodic transformation  is conservative,
 we deduce from (1) that  subset of conservative transformations is $d_2$-dense
 in
 $\text{{\rm Aut}}_2(X,\mathcal{A},\mu)$.
\end{proof}

 We say that a transformation $T\in  \text{{\rm Aut}}_{\widehat{0}}(X,\mathcal{A},\mu)$
 is \textit{locally aperiodic} if there exists $A\in\mathcal{A}$ of positive finite measure such that $T\in\text{Aut}_2(A,\mathcal{A}\cap A,\mu|_A)$ and $T|_A$ is aperiodic.

\begin{prop}\label{prop: Rokhlin property}
Let $T\in   \text{{\rm Aut}}_{\widehat{0}}(X,\mathcal{A},\mu)$
 be locally aperiodic. 
 Then the conjugacy class of $T$ is $d_2$-dense in   \text{{\rm Aut}}$_2(X,\mathcal{A},\mu)$.
  In particular,   \text{{\rm Aut}}$_2(X,\mathcal{A},\mu)$ has the Rokhlin property. 	
\end{prop}
\begin{proof}
In view of Proposition \ref{prop: local are dense in Aut2}, it is enough to show that each local transformation  is in the $d_2$-closure of the conjugacy class of $T$. 
Select a subset  $A\in\mathcal{B}$ of positive finite measure such that $T|_A$ is aperiodic. It follows from \cite[Theorem 2]{ChKa79} that $\left\{S^{-1}TS:\ S\in \text{Aut}_2(A,\mathcal{A}\cap A,\mu|_A)\right\}$ is $d_2$-dense in $\text{Aut}_2(A,\mathcal{A}\cap A,\mu|_A)$. Since for every subset $B\subset X$ of positive finite measure, the map 
\[
\text{Aut}_2(A,\mathcal{A}\cap A,\mu|_A)\ni T\mapsto \left(\tau_{B,A}\right)^{-1}T\tau_{B,A} 
\in \text{Aut}_2(B,\mathcal{A}\cap B,\mu|_B) 
\]
is a topological group isomorphism 
and $\tau_{B,A}\in \text{Aut}_2(X,\mathcal{A},\mu)$, every local transformation is in the $d_2$-closure of the conjugacy class of $T$. 
\end{proof}

We now state one of the main result of this section.

\begin{thm}\label{thm: main category thm}
The   subset 
	 $$
	 \mathrm{Erg}^{\mathrm{III}_1}_2(X,\mathcal A,\mu):=\{T\in\mathrm{Aut}_2(X,\mathcal A,\mu): T\text{ is ergodic of type } \mathrm{III}_1\}
	 $$
	  is a dense $G_\delta$ in $\text{{\rm Aut}}_2(X,\mathcal A,\mu)$.  
\end{thm}
The proof of Theorem \ref{thm: main category thm} relies on the method of inducing which we now describe. 
Given a  transformation $T\in \text{Aut}(X,\mathcal{A},\mu)$, a subset $A\in\mathcal{A}$ is called {\it $T$-sweeping out} if 
$\mu\left(X\setminus\bigcup_{n=1}^\infty T^{-n}A\right)=0$.
If $T$ is ergodic then each subset $A\in\mathcal A$ of positive measure is 
$T$-sweeping out.
Given a $T$-sweeping out  subset $A$, we can define the \textit{induced map}, also known as the first return map, to $A$ as  a nonsingular transformation $T_A$ of the space $(A,\mathcal{A}\cap A,\mu|_A)$ defined on a full measure subset of $A$ by 
\[
T_Ax:=T^{\varphi_A(x)}x,
\]
where $\varphi_A(x):=\inf\left\{n\in\mathbb{N}:\ T^nx\in A\right\}$ is the \textit{first return time function} to $A$. 
We note  $T_A$ is well defined because $A$ is $T$-sweeping out.
The following facts will be used in the proof of Theorem \ref{thm: main category thm}:
\begin{itemize} 
\item If  $A$ is $T$-sweeping out then $T$ is ergodic   if and only if $T_A$ is ergodic \cite[Proposition 1.5.2]{Aar97InfErg}.
\item
If $T$ is ergodic and $A$ is of positive measure
 then $T$ is of type III$_1$   if and only if $T_A$ is of type III$_1$.\footnote{This can be 
 proved via inducing in the Maharam extension to the subset $A\times \mathbb{R}$ and noticing that $A$ is $T$-sweeping out if and only if $A\times \mathbb{R}$ is sweeping out for the Maharam extension of $T$.}
\end{itemize}

 \begin{proof}[Proof of Theorem {\rm \ref{thm: main category thm}}]
 Since Erg$^{\mathrm{III}_1}(X,\mathcal A,\mu)$ is a $G_\delta$ in Aut$(X,\mathcal A,\mu)$ and $d_2$ is stronger than the weak topology, it follows that 
 Erg$^{\mathrm{III}_1}_2(X,\mathcal A,\mu)$ is a $G_\delta$ in $\mathrm{Aut}_2(X,\mathcal A,\mu)$.
 It remains to show that  Erg$^{\mathrm{III}_1}_2(X,\mathcal A,\mu)$ is dense in $\mathrm{Aut}_2(X,\mathcal A,\mu)$.
 For that, take a local transformation $T\in \mathrm{Aut}_2(X,\mathcal A,\mu)$.
 Our purpose is to find a sequence of transformations from Erg$^{\mathrm{III}_1}_2(X,\mathcal A,\mu)$ that converges to $T$ in $d_2$ and apply Proposition~\ref{prop: local are dense in Aut2}.
 Let $A$ be a subset of finite measure such that $T|_{X\setminus A}=\text{id}_{X\setminus A}$.
 Take a sequence  $\{B_n\}_{n=1}^\infty$ of  subsets of finite  measure and 
 a
 sequence $\{\mathcal P_n\}_{n=1}^\infty$ of countable partitions of $X\setminus A$ into subsets of finite measure such that the following are satisfied:
 \begin{itemize}
 \item
 $A\subset B_n\subset B_{n-1}$,  $\mu(B_n\setminus A)>0$ for each $n$ and $\lim_{n\to\infty}\mu(B_n)=\mu(A)$,
 \item
 $\mathcal P_n=\{p^{(n)}_{l,j}: l\in\Bbb N, 1\le j\le J_l^{(n)}\}$ for some $J_l^{(n)}>n$ for all $n,l>0$,
  \item
 $\mu(p^{(n)}_{l,1})=\mu(p^{(n)}_{l,2})=\cdots=\mu(p^{(n)}_{l,J_l^{(n)}})$ for all $n,l>0$,
 \item
 $B_n\setminus A=\bigsqcup_{l=1}^\infty p^{(n)}_{l,1}$,
  \item
  if $\mathcal Q_n:=\Big\{\bigsqcup_{j=1}^{J^{(n)}_l}p^{(n)}_{l,j}: l\in\Bbb N\Big\}$ then 
  $\mathcal Q_1\prec \mathcal Q_2\prec\cdots$
 and   $\mathcal Q_n\to
  \mathcal A|_{X\setminus A}$ as $n\to\infty$.

 \end{itemize}
 Fix also a sequence $\{S_n\}_{n=1}^\infty$ of  transformations $S_n\in\mathrm{Erg}^{\mathrm{III}_1}(B_n,\mathcal A\cap B_n,\mu |B_n)$ that weakly converges to $T|_A$.
 By this we mean that  $\|U_{S_n}f-U_Tf\|_2\to 0$ as $n\to\infty$ for each $f\in L^2(A)$. 
 We now can construct, for each $n>0$, a nonsingular transformation $T_n$ of $X$ satisfying the following
 conditions:
 \begin{itemize}
 \item
$ T_nx=S_nx$ for all $x\in A$,
 \item
$
T_n p^{(n)}_{l,j}=
 \begin{cases}
 p^{(n)}_{l,j+1}, & \text{if }j\ne J_l^{(n)}\\
 S_np^{(n)}_{l,1}, &\text{if }j= J_l^{(n)}
\end{cases}
 $
 \item
 $T_n'(x)=1$ for each $x\not\in \bigsqcup_{l=1}^{\infty}S_np^{(n)}_{l,1}$.
 \end{itemize}
 It follows straightforwardly  from the definition of $T_n$ that 
 $B_n$ is $T_n$-sweeping out and
 $S_n$ is induced by $T_n$. Hence $T_n$ is ergodic of type $\mathrm{III}_1$ for each $n\in\Bbb N$.
 Of course, $T_n\in\text{Aut}_2(X,\mathcal A,\mu)$.
 We claim that $T_n\to T$ in $d_2$ as $n\to\infty$.
 Take an atom $q$ of $\mathcal Q_n$.
 Then $q=\bigsqcup_{j=1}^{J^{(n)}_l}p^{(n)}_{l,j}$ for some $l>0$ and
 $$
 \|U_T1_q-U_{T_n}1_q\|_2=\bigg\|1_{p_{l,1}^{(n)}}- U_{T_n}1_{p_{l,J_l^{(n)}}^{(n)}}\bigg\|_2
 \le 2\Big\|1_{p_{l,1}^{(n)}}\Big\|_2={\frac{2\|1_q\|_2}{\sqrt{J_l^{(n)}}}}.
 $$
Hence for each  function $f\in L^2(X\setminus A)$ which is $\mathcal Q_m$-measurable for some $m>0$, we have that
$$
\|U_Tf-U_{T_n}f\|_2\le \frac{2\|f\|_2}{\sqrt n}
$$
 whenever $n\ge m$.
 Hence $U_{T_n}f\to U_{T}f$ as $n\to\infty$.
 Since $\mathcal Q_n\to
  \mathcal A|_{X\setminus A}$ as $n\to\infty$, it follows that
  $U_{T_n}f\to U_Tf$ as $n\to\infty$ for each  function $f\in L^2(X\setminus A)$.
On the other hand,
$U_{T_n}g=U_{S_n}g$ for each $g\in L^2(A)$ and $U_{S_n}g\to U_Tg$ weakly as $n\to\infty$.
We  deduce that $T_n\to T$ weakly as $n\to\infty$.
  Since $T_n'(x)\ne 1$ only if $x\in S_nA\sqcup\bigsqcup_{l=1}^\infty S_np^{(n)}_{l, 1}=S_nB_n=B_n$, we obtain that
  $$
  \begin{aligned}
  \big\|\sqrt{T'}-\sqrt{T_n'}\big\|_2&= \big\|(\sqrt{T'}-\sqrt{T_n'})1_{B_n}\big\|_2\\
  &=
  \|U_T1_{B_n}-U_{T_n}1_{B_n}+\sqrt{T_n'}(1_{T_nB_n}-1_{B_n})\|_2\\
  &\le  \|(U_T-U_{T_n})1_{B_n}\|_2+\sqrt{\mu(B_n\triangle T_n^{-1}B_n)}
  \end{aligned}
  $$
  and $\mu(B_n\triangle T_n^{-1}B_n)=\mu\Big(\Big(\bigsqcup_{l=1}^\infty p^{(n)}_{l,1}\Big)\triangle \Big(\bigsqcup_{l=1}^\infty p^{(n)}_{l,J_l^{(n)}}\Big)\Big)=2\mu(B_n\setminus A)\to 0$
  as $n\to\infty$.
  It  follows that $ \big\|\sqrt{T'}-\sqrt{T_n'}\big\|_2\to 0$, as desired.
 \end{proof}

 \begin{rem} Let $B$ be a subset of positive finite measure in the standard $\sigma$-finite non-atomic measure space $(X,\mathcal A,\mu)$.
 Arguing as in the proof of the above theorem, we see that given an ergodic nonsingular transformation $S$ of $(B,\mathcal A
 \cap B,\mu|_B)$, we can construct an ergodic transformation $T\in\mathrm{Aut}_2 (X,\mathcal A,\mu)$ such that $S=T_B$.
 It is well known that  given an ergodic nonsingular flow $W$ (i.e. an $\Bbb R$-action $(W(t))_{t\in\Bbb R}$), there is 
 an ergodic nonsingular transformation whose associated flow is isomorphic to $W$.
On the other hand, 
 the associated flow of $T$ is isomorphic to the associated flow of each transformation induced by $T$.
 It  follows  from these facts that given an ergodic nonsingular flow $W$, there is 
 $T\in\mathrm{Aut}_2 (X,\mathcal A,\mu)$ such that the associated flow of $T$ is $W$.
 In particular, for each $\lambda\in[0,1]$, the group
 $\mathrm{Aut}_2 (X,\mathcal A,\mu)$ contains an ergodic transformation of Krieger's type III$_\lambda$.
 For the definition of the associated flow,  Krieger's type and other concepts of orbit theory we refer to
 \cite{DaSi}.
  \end{rem}

The following assertion is an analogue of Proposition~\ref{prop: local are dense in Aut2}
for Aut$_1(X,\mathcal A,\mu)$ furnished with  $d_1$.
It can not hold for the entire group Aut$_1(X,\mathcal A,\mu)$ 
because $\text{Aut}_{\widehat{0}}(X,\mathcal{A},\mu)$ is a subgroup of the proper closed subgroup Ker$\,\chi$ of  Aut$_1(X,\mathcal A,\mu)$.
However, it holds 
for
 $\mathrm{Ker}\,\chi$.

\begin{prop}\label{Aut_1 version}
$\mathrm{Aut}_{\widehat{0}}(X,\mathcal{A},\mu)$  is $d_1$-dense in $\mathrm{Ker}\,\chi$. 
\end{prop}	
\begin{proof}
Let $T\in \text{Ker}\,\chi$. 
Then there exists an increasing sequence $A_n\in\mathcal{A}$ of finite measure subsets satisfying $\bigcup_{n=1}^\infty A_n=X$,  
\[
\int_{X\setminus A_n}\left|\frac{d\mu\circ T^{-1}}{d\mu}-1\right|d\mu:=\epsilon_n\xrightarrow[]{n\to\infty}0.
\] 
and 
\[
\left|\mu\left(T^{-1}A_n\right)-\mu\left(A_n\right)\right|<\epsilon_n.
\]
Select subsets $B_n\subset X$ of finite measure such that $A_n\cup T^{-1}A_n\subset B_n$.
As in the proof of Proposition \ref{prop: local are dense in Aut2}, we define a transformation $T_n\in\mathrm{Aut}_{\widehat{0}}(X,\mathcal{A},\mu)$ by setting
$$
T_n^{-1}x:=\begin{cases}
T^{-1}x, & x\in A_n\\
\tau_{B_n\setminus A_n,B_n\setminus T^{-1}A_n}x, & x\in B_n\setminus A_n\\
x, & x\notin B_n.
\end{cases}
$$
We will now show that $T_n$ converges in $d_1$ to $T$ as $n\to\infty$.
 Since $T_n^{-1}x=T^{-1}x$ for $x\in A_n$, it follows that $T_n$ converges weakly to $T$. 
 Next, as in the proof of Proposition \ref{prop: local are dense in Aut2}, we see that
 (\ref{RN for T_n}) holds and hence 
\begin{align*}
\left\Vert \frac{dm\circ T^{-1}}{d\mu}-\frac{d\mu\circ T_n^{-1}}{d\mu}\right\Vert_1 &\leq \int _{X\setminus A_n}\left|\frac{d\mu\circ T^{-1}}{d\mu}-1\right|d\mu\\
&+\int\left|\frac{\mu\left(B_n\setminus T^{-1}A_n\right)}{\mu\left(B_n\setminus A_n\right)}-1\right|d\mu\\
&\leq \epsilon_n+\left|\mu\left(T^{-1}A_n\right)-\mu\left(A_n\right)\right|\leq 2\epsilon_n.
\end{align*}
\end{proof}
Proceeding from here in a similar way as we have done  for Aut$_2(X,\mathcal A,\mu)$ but utilizing Proposition~\ref{Aut_1 version} instead of Proposition~\ref{prop: local are dense in Aut2},  we arrive at the 
following analogues of Corollary~\ref{corol},  Proposition~\ref{prop: Rokhlin property}  and Theorem~\ref{thm: main category thm} for $\text{Ker}\,\chi$.

\begin{thm}\label{thm: 2nd main category thm}\
\begin{enumerate}
\item
The subset of periodic transformations from $\text{{\rm Aut}}_{\widehat{0}}(X,\mathcal{A},\mu)$ is a $d_1$-dense subset of  $\text{{\rm Ker}}\,\chi$.
\item
For each locally aperiodic transformation $T$,
the conjugacy class of $T$  (in  $\text{{\rm Aut}}_{1}(X,\mathcal{A},\mu)$) is a $d_1$-dense subset of $\text{{\rm Ker}}\,\chi$.
In particular,  $\text{{\rm Ker}}\,\chi$ has the Rokhlin property. 
\item
The following three sets:
$$
\begin{aligned}
&\{T\in\mathrm{Aut}_1(X,\mathcal A,\mu): T\text{ is conservative}\}, \\
& \{T\in\mathrm{Aut}_1(X,\mathcal A,\mu): T\text{ is ergodic}\} \text{ and }\\
&\{T\in\mathrm{Aut}_1(X,\mathcal A,\mu): T\text{ is  of type } \mathrm{III}_1\} 
\end{aligned}
$$
are dense $G_\delta$-subsets  of  $\text{{\rm Ker}}\,\chi$ endowed with $d_1$.
\end{enumerate}
\end{thm}

\section{Basic dynamical properties of Poisson suspensions for  locally compact group actions}\label{sec: Basic dynamics}

\subsection{Unitary and affine  Koopman representations}
Let $G$ be a locally compact non-compact  second countable group.
A nonsingular $G$-action on a standard $\sigma$-finite measure space $(X,\mathcal A,\mu)$ is a Borel map
$$
G\times X\ni (g,x)\mapsto T_gx\in X
$$
such that the mapping $T_g:X\ni x\mapsto T_gx\in X$ is a $\mu$-nonsingular bijection of $X$ for each $g\in G$.
Equivalently, a nonsingular $G$-action can be defined as a continuous group homomorphism $T:G\ni g\mapsto T_g\in\mathrm{Aut}(X,\mathcal A,\mu)$, where the later group is furnished with the weak topology.
The corresponding  unitary Koopman representation $G\ni g\mapsto U_{T_g}\in\mathcal U_\Bbb R(L^2(\mu))$
of $G$ is continuous in the weak (and strong) operator topology.
Since the real Hilbert space  $L^2_\Bbb R(\mu)$ is invariant under the unitary Koopman representation, we can consider $L^2_\Bbb R(\mu)$
  as a $G$-module.
 Denote by $Z^1(G, L^2_\Bbb R(\mu) )$ the vector space of all continuous 1-cocycles of $G$ in $L^2_\Bbb R(\mu)$, i.e. the  mappings $c:G\ni g\mapsto c(g)\in L^2_\Bbb R(\mu)$ such that $c(gh)=c(g)+ U_{T_g}c(h)$ for all $g,h\in G$.
By $B^1(G, L^2_\Bbb R(\mu) )$ we denote the subspace of 1-coboundaries, i.e. those 1-cocycles
$c\in Z^1(G, L^2_\Bbb R(\mu) )$ for which there is $f\in L^2_\Bbb R(\mu)$ such that $c(g)=U_{T_g}f-f$.
We also recall that a 1-cocycle $c\in Z^1(G, L^2_\Bbb R(\mu) )$ is called {\it proper} if $\|c(g)\|_2\to\infty$ as $g\to\infty$.

Suppose now that we are given a nonsingular $G$-action $T$ such that  $T_g\in\mathrm{Aut}_2(X,\mathcal A,\mu)$ for each $g\in G$.
Since the $d_2$-topology is stronger than the weak topology, it follows that the the restriction (to $\mathrm{Aut}_2(X,\mathcal A,\mu)$) of the Borel structure generated by the weak topology coincides with the Borel structure generated by $d_2$.
Hence the map $T$ considered as a homomorphism from $G$ to $\mathrm{Aut}_2(X,\mathcal A,\mu)$
is $d_2$-Borel.
Since each Borel homomorphism from a Polish group to another Polish group is continuous, we obtain that $T$ is continuous as a map from $G$
to $\mathrm{Aut}_2(X,\mathcal A,\mu)$ furnished with $d_2$.
We recall that $\mathrm{Aut}_2(X,\mathcal A,\mu)$  embeds  continuously into Aff$_\Bbb R(L^2(\mu))$ via the 
affine Koopman representation $A^{(2)}$ (see Definition~\ref{def:affine Koopman}).
Thus, we obtain the following proposition.

\begin{prop} If $T$ is a nonsingular $G$-action such that
$T_g\in\mathrm{Aut}_2(X,\mathcal A,\mu)$ for each $g\in G$ then the 1-cocycle
$$
c_T:G\ni g\mapsto c_T(g):=\sqrt{T_g'}-1\in L^2_\Bbb R(\mu)
$$
 of $G$  in $L^2_\Bbb R(\mu)$ is continuous. 
\end{prop}

Thus, $c_T\in Z^1(G,L^2_\Bbb R(\mu))$. 
It follows that under the condition of the above proposition, 
 a (weakly) continuous affine representation $A_T:G\ni g\mapsto A_{T_g}\in \mathrm{Aff}_\Bbb R(L^2(\mu))$ of $G$ in $L^2(\mu)$ is well defined by the restriction of $A^{(2)}$ to $G$, i.e.
$$
A_{T_g}h:=U_{T_g}h+ c_T(g).
$$

\begin{defn}\label{def:affine}
We call $A_T$ {\it the affine Koopman} representation of $G$ generated by $T$.
\end{defn}

In the next proposition we compare the property to have a fixed vector for the unitary and affine Koopman representations.

\begin{prop}\label{prop:fixed point}\ 
\begin{enumerate}
\item
The unitary Koopman representation has a non-trivial fixed vector if and only if
$T$ admits a non-trivial absolutely continuous invariant probability measure.
\item
The affine Koopman representation has a fixed vector if and only if $T$ admits an absolutely continuous invariant measure belonging to $\mathcal M^+_{\mu,2}$.
\end{enumerate}
\end{prop}

\begin{proof} (1) is standard. We leave its proof to the reader.

(2) If there is $h\in L^2(\mu)$ such that $A_gh=h$ for all $g\in G$ then
$$
\sqrt{ T_g'}\,(h\circ T_g^{-1}+1)=h+1.
$$
We now define  a measure $\nu$ which is absolutely continuous with respect to $\mu$ and
such that $\frac{d\nu}{d\mu}:=(h+1)^2$.
Then $\nu$ is invariant under $T$.
Since $\sqrt{\frac{d\nu}{d\mu}}-1=|h+1|-1\in L^2_\Bbb R(\mu)$, it follows that $\nu\in \mathcal M^+_{\mu,2}$, as desired.
The converse assertion is proved in a similar way by ``reversing'' the argument.
 \end{proof}

\subsection{Existence of an equivalent probability measure}

We now examine when the Poisson suspension $T_*:=\{(T_g)_*\}_{g\in G}$ of $T$ admits a  $\mu^*$-absolutely continuous  invariant probability measure.

\begin{prop}
\label{prop:coboundary}
Let  $T:=\left\{ T_{g}\right\} _{g\in G}$
be a nonsingular $G$-action such that $T_{g}\in\mathrm{Aut}_{2}(X,\mathcal{A},\mu)$
for every $g\in G$.
Then the following assertions are
equivalent:
\begin{enumerate}
\item The 1-cocycle $c_T$ is bounded, i.e. there is $d>0$ such that
$\|c_T(g)\|_2\le d$ for each $g\in G$.
\item $c_T\in B^1(G,L^2_\Bbb R(\mu))$.
\item There exists a  $T$-invariant measure $\nu\in\mathcal{M}_{\mu,2}^{+}$
(hence the probability measure $\nu^{*}$ is  $T_{*}$-invariant 
and $\nu^{*}\ll\mu^{*}$).
\item There exists a $T_{*}$-invariant probability measure $\rho\ll\mu^{*}$.
\end{enumerate}
\end{prop}

\begin{proof}
(1) $\Longleftrightarrow$ (2) is classical, see \cite[Proposition 2.2.9]{BHV08Kazh},
 for a proof. 
 
 (2) $\implies$ (3) because if $c_T(g)=U_{T_g}h-h$ for some $h\in L^2_\Bbb R(\mu)$
 and all $g\in G$ then $h$ is a fixed vector for the affine Koopman representation of $G$ generated by $T$.
 It remains to apply Proposition~\ref{prop:fixed point}(2).
 
  (3) $\implies$(4) is obvious if we set $\rho:=\nu^*$.
 
 (4) $\implies$ (1)
 It follows from ({\ref{was}}) that  for each $g\in G,$
 $$
\sqrt{   \Exp{T^{\prime}_g-1}}=e^{-\frac12\|c_T(g)\| _{2}^{2}}\,\Exp{\sqrt{T^{\prime}_g}-1}.
 $$
 Integrating this equality and using Corollary~\ref{cor:R-N} we obtain that
 \begin{align}\label{eq:c_T}
\langle U_{(T_g)_*}1,1\rangle_{L^2(\mu^*)}
&=e^{-\frac 12\|c_T(g)\|_2^2}.
\end{align}
Therefore if $c_T$ were unbounded then there would exist  a sequence
$\{g_n\}_{n=1}^\infty$ of elements of  $G$ such that
$\langle U_{(T_g)_*}1,1\rangle_{L^2(\mu^*)}
\to 0$ as $n\to\infty$.
Hence for all subsets $A,B$ of finite measure in $(X^*,\mu^*)$, we have that
$$
\left\langle U_{(T_{g_n})_*}1_A,1_B\right\rangle_{L^2(\mu^*)}\le
\langle U_{(T_g)_*}1,1\rangle_{L^2(\mu^*)}
\to 0 \quad  \text{(as $n\to\infty$)}.
$$
This implies, in turn, that $U_{(T_{g_n})_*}h\to 0$  for each $h\in L^2(X^*,\mu^*)$ in the weak topology. 
However, taking $h:=\sqrt{\frac{d\rho}{d\mu^*}}$, we obtain that
$U_{(T_{g_n})_*}h=h$, a contradiction.
\end{proof}

\begin{cor}
Assume that $\mu$ is infinite and  $T$ is ergodic. 
There exists
a $T_{*}$-invariant probability measure $\rho\sim\mu^{*}$ if and only
if there exists a $T$-invariant measure $\nu\in\mathcal{M}_{\mu,2}^{\circ,+}$.
In this case $\rho=\nu^{*}$ and $T_{*}$ is ergodic (and weakly mixing).
\end{cor}

\begin{proof}
The ergodicity of $T_{*}$ when $\nu^{*}$ is invariant follows from
the lack of $T$-invariant set of non-zero and finite $\nu$-measure
which is the classical ergodicity criteria in the measure preserving case.
\end{proof}

\subsection{Conservativeness and zero type for Poisson $G$-actions}
We first recall some standard definitions for locally compact group actions.

\begin{defn} Let $T=\{T_g\}_{g\in G}$  be a  nonsingular $G$-action on $(X,\mathcal A,\mu)$.
\begin{itemize}
\item
$T$ is called {\it conservative} if for each subset $A$ of positive measure and a compact subgroup $K\subset G$, there is $g\in G\setminus  K$ such that
$\mu(A\cap T_g A)>0$.
\item
$T$ is called {\it dissipative} if it is not conservative.
\item 
$T$ is called {\it totally dissipative} if the restriction of $T$ to every $T$-invariant subset of $X$ is dissipative.
\item
$T$ is called {\it of zero type} if $U_{T_g}\to 0$ as $g\to\infty$ in the weak operator topology.
\end{itemize}
\end{defn}
In case $G=\Bbb Z$, these definitions for conservativeness  and dissipativeness are equivalent to those given above.
We will need the following lemma from \cite[Proposition A.34]{ArIsMa}.

\begin{lem}\label{lem:aux} Let $\lambda$ be a left Haar measure on $G$.
If there is $s\in\Bbb R$ such that
$\int_G(T_g')^s\,d\lambda(g)<\infty$ then $T$ is totally dissipative.
\end{lem}

The next corollary follows from Lemma~\ref{lem:aux} and (\ref{eq:c_T}).

\begin{cor}\label{co:total dissipat}
 If $T=\{T_g\}_{g\in G}$ is a nonsingular $G$-action such that
$T_g\in\mathrm{Aut}_2(X,\mu)$ for each $g\in G$ and 
$$
\int_G e^{-\frac12\|c_T(g)\|_2^2} \,d\lambda(g)<\infty
$$
 then $T_*$ is totally dissipative.
\end{cor}

Let $T_g\in\mathrm{Aut}_2(X,\mu)$ for all $g\in G$ and let $A_T=\{A_{T_g}\}_{g\in G}$ stand for the affine Koopman representation of $G$ generated by $T$.

\begin{prop}\label{prop:mixing_proper} 
$T_*$ is of zero type if and only if $c_T$ 
is proper.
\end{prop}
\begin{proof}
We first note that $T_*$ is of zero type if and only if $\langle U_{(T_g)_*}1,1\rangle\to 0$ as $g\to\infty$.
This fact was shown (for arbitrary nonsingular $G$-actions)
 in the proof of Proposition~\ref{prop:coboundary}.
 It remains to apply (\ref{eq:c_T}).
\end{proof}
A class of groups {\it having property} (BP$_0$) was introduced in  \cite{CoTeVa2}.
This class includes  groups with property (T), solvable groups (in particular, Abelian groups),  connected Lie groups, 
 linear algebraic groups over a local field of characteristic zero, etc.
We need only the following fact: if $G$ has property (BP$_0$) and $T$ is of zero type then $c_T$ is either proper or bounded \cite{CoTeVa2}.
This fact and Propositions~\ref{prop:mixing_proper} and \ref{prop:coboundary} imply the next corollary.

\begin{cor}\label{cor:0type}
 Let $G$ have property {\rm (BP$_0$)}. 
Let  $T=\left\{ T_{g}\right\} _{g\in G}$
be a nonsingular $G$-action such that $T_{g}\in\mathrm{Aut}_{2}(X,\mathcal{A},\mu)$
for every $g\in G$.
Suppose that there is no any $T$-invariant measure in $\mathcal M_{\mu,2}^+$.
If $T$ is of zero type   then $T_*$ is of zero type.
\end{cor}

\section{Furstenberg entropy and Stationarity}\label{sec: F-entropy}

\subsection{Furstenberg entropy}
Let $G$ be a locally compact group second countable group and let $\kappa$ be a generating
probability measure on $G$ (i.e. the support of $\kappa$ generates a dense subgroup
of $G$). The Furstenberg $\kappa$-entropy of a non-singular action
$S=\left(S_{g}\right)_{g\in G}$ on a probability space $\left(Y,\mathcal{B},\nu\right)$
is the quantity
\[
h_{\kappa}(S,\nu):=-\int_{G}\left(\int_{Y}\log S_{g}^{\prime}\,d\nu\right)d\kappa(g)\in\left[0,+\infty\right].
\]

We note that $h_{\kappa}(S,\nu)=0$ if and only if $S$ preserves $\nu$.
We will need the following corollary from Theorems~\ref{prop:LogT'} and \ref{thm: LogT'}.

\begin{cor}
\label{prop:Log}Let $T$ be in $\text{{\rm Aut}}_{2}(X,\mathcal{A},\mu)$.
We have that
\[
\mathbb{E}_{\mu^{*}}[\log\left(T_{*}\right)^{\prime}]=-\int_{X}\left(T^{\prime}-1-\log T^{\prime}\right)d\mu,
\]
it is finite if and only if 
$\int_{\{x\in X:\,|\log T^{\prime}(x)|>1\}}|\log T^{\prime}|d\mu<\infty$.
In particular, if $T\in\text{{\rm Aut}}_{1}(X,\mathcal{A},\mu)$,
then $\int_{X}\log T^{\prime}d\mu$ is well defined and takes its
value in the extended interval $\left[-\infty,\chi\left(T\right)\right]$. 
We get in this
case:
\[
\mathbb{E}_{\mu^{*}}[\log(T_{*})^{\prime}]=-\chi(T)+\int_{X}\log T^{\prime}d\mu.
\]
\end{cor}

Given $c>0$, we denote by $\mu_c$ the $c$-scaling of $\mu$, i.e. $\mu_c(B)=c\mu(B)$ for each $B\in\mathcal A$.

\begin{cor}
\label{cor:entr-c}
Let $T$ be in $\text{{\rm Aut}}_{2}(X,\mathcal{A},\mu)$. 
Then for each $c>0$, the following
formula holds:
\[
\mathbb{E}_{\mu_c^{*}}\left[
\log\frac{d\mu_c^{*}\circ T_{*}^{-1}}{d\mu_c^{*}}\right]=c\,\mathbb{E}_{\mu^{*}}\left[\log\frac{d\mu^{*}\circ T_{*}^{-1}}{d\mu^{*}}\right].
\]
\end{cor}

\begin{proof}
Since $\frac{d\mu_c\circ T^{-1}}{d\mu_c}=\frac{d\mu\circ T^{-1}}{d\mu}$,
we deduce from Corollary~\ref{prop:Log} that
\begin{align*}
\mathbb{E}_{\mu_c^{*}}\left[\log\frac{d\mu_c^{*}\circ T_{*}^{-1}}{d\mu_c^{*}}\right]   & 
=-c\int_{X}\left(\frac{d\mu\circ T^{-1}}{d\mu}-1-\log\frac{d\mu\circ T^{-1}}{d\mu}\right)d\mu\\
 & =c\mathbb{E}_{\mu^{*}}\left[\log\frac{d\mu^{*}\circ T_{*}^{-1}}{d\mu^{*}}\right].
\end{align*}
\end{proof}
In view of the above results, we obtain the following.

\begin{prop}\label{prop:entr}
Let $T=(T_{g})_{g\in G}$
be a non-singular $G$-action on $(X,\mathcal{A},\mu)$ such that $T_g \in\text{{\rm Aut}}_{2}(X,\mathcal{A},\mu)$ for each $g\in G$. 
Then:
\[
h_{\kappa}(T_{*},\mu^{*})=-\int_{G}\left(\int_{X}\left(\log T_{g}^{\prime}-T_{g}^{\prime}+1\right)d\mu\right) d\kappa(g).
\]
Moreover $h_{\kappa}\left(T_{*},\mu_c^{*}\right)=ch_{\kappa}\left(T_{*},\mu^{*}\right)$,
for any $c>0$.
If $T_{g}\in\text{{\rm Aut}}_{1}(X,\mathcal{A},\mu)$ for each $g\in G$
then
\[
h_{\kappa}\left(T_*,\mu^{*}\right)=\int_{G}\left(\chi\left(T_{g}\right)-\int_{X}\log T_{g}^{\prime}\,d\mu\right) d\kappa(g).
\]
\end{prop}

\begin{cor}\label{cor:entr-cons}
Let $T=(T_{g})_{g\in G}$
be a non-singular $G$-action on $(X,\mathcal{A},\mu)$ such that $T_g \in\text{{\rm Aut}}_{1}(X,\mathcal{A},\mu)$ for each $g\in G$. 
If one of the following conditions  is satisfied:
\begin{enumerate}
\item  $T_{g}$ is conservative for each $g\in G$,
\item $\kappa$ is symmetric,
\end{enumerate}
then we have
\[
h_{\kappa}\left(T_{*},\mu^{*}\right)=-\int_{G}\left(\int_{X}\log T_{g}^{\prime}\,d\mu\right)
d\kappa(g).
\]
\end{cor}

\begin{proof}
The first point follows from Proposition \ref{prop:Conservativity Aut1}.
For the second, take an increasing sequence of compact symmetric sets $\left\{ A_{n}\right\} _{n\in\mathbb{N}}$
such that $G=\bigcup_{n\in\mathbb{N}}A_{n}$.
Then

\[
\int_{G}\left(\chi\left(T_{g}\right)-\int_{X}\log T_{g}^{\prime}\,d\mu\right)d\kappa(g)=\lim_{n\to\infty}\int_{A_{n}}\left(\chi\left(T_{g}\right)-\int_{X}\log T_{g}^{\prime}d\mu\right)d\kappa(g
).
\]
Since $G\ni g\mapsto\chi\left(T_{g}\right)\in\Bbb R$ is a continuous group homomorphism, it is integrable on $A_{n}$.
Moreover, $\chi(T_g^{-1})=-\chi(T_g)$ and hence, by the
symmetry of $\kappa$ and $A_n$, we obtain that $\int_{A_{n}}\chi(T_{g})d\kappa(g)=0$.
The result follows.
\end{proof}

\subsection{Stationarity for Poisson suspensions}

Let $\kappa$ be a generating probability measure on $G$.
 We recall that a non-singular
action $\left(S_{g}\right)_{g\in G}$ on a probability space $\left(Y,\mathcal{B},\nu\right)$
is called {\it $\kappa$-stationary} if $\int_{G}\nu\circ S_{g}\,d\kappa(g)=\nu$.

\begin{prop}\label{prop:statio}
Let $T=(T_{g})_{g\in G}$
be a non-singular $G$-action on $(X,\mathcal{A},\mu)$ such that $T_g \in\text{{\rm Aut}}_2(X,\mathcal{A},\mu)$ for each $g\in G$. 
 The Poisson suspension  $T_*=((T_{g})_{*})_{g\in G}$ of\, $T$ acting on
  the space $\left(X^{*},\mathcal{A}^{*},\mu^{*}\right)$
 is stationary if
and only if $T_*$ preserves $\mu^*$.
\end{prop}

\begin{proof}
Assume  that $T_*$ is $\kappa$-stationary. 
Then
 $$
 \mu^{*}=\int_{G}\mu^{*}\circ\left(T_{g}\right)_{*}d\kappa(g)
 =\int_{G}\left(\mu\circ T_{g}\right)^{*}d\kappa(g).
 $$
Computing the intensity of the two sides of this equation we obtain that
\begin{equation}\label{eq:intensities}
\mu=\int_{G}\mu\circ T_{g}\,d\kappa(g).
\end{equation}
Take a set $A\in\mathcal{A}$ with $0<\mu\left(A\right)<+\infty$
and compute the void probabilities of $A$ using $\mu^{*}$ and $\int_{G}\left(\mu\circ T_{g}\right)^{*}d\kappa(g)$.
We obtain that
\[
e^{-\mu\left(A\right)}=\int_{G}e^{-\mu\circ T_{g}\left(A\right)}d\kappa(g).
\]
From this and the Jensen inequality we deduce that
$$
\mu\left(A\right)=-\log\int_{G}e^{-\mu\circ T_{g}\left(A\right)}d\kappa(g)\le \int_{G}
\mu\circ T_{g}(A)d\kappa(g)
$$
In view of \eqref{eq:intensities} we see that  the equality case takes place in the Jensen inequality.
This happens if and only if 
 $\mu\left(A\right)=\mu(T_{g}A)$
for $\kappa$-a.e. $g\in G$. 
Thus $\mu=\mu\circ T_{g}$ for $\kappa$-a.e.
$g\in G$. 
Since $\kappa$ is generating, $\mu$ is invariant under $T_g$ for $g$ belonging to a dense
subgroup in $G$.
 Therefore $\mu=\mu\circ T_{g}$ for every $g\in G$.
\end{proof}

\section{Property (T), Poisson Suspensions and Furstenberg entropy}\label{sec: Property T}
Let $G$ be a locally compact second countable group. The group has Kazhdan property (T) if every unitary representation of $G$ which has almost invariant vectors admits a non zero invariant vector. There are numerous equivalent characterizations of property (T), among them is that for every unitary representation of $G$ every  cocycle  is a coboundary \cite{BHV08Kazh}. Theorem \ref{th:T-prop} is a new characterization of property (T) in terms of nonsingular Poisson suspensions.

\begin{defn}
A nonsingular  action $R=\{R_g\}_{g\in G}$ of $G$ is called {\it Poisson} if there is a $\sigma$-finite standard measure space $(X,\mathcal A,\mu)$ and a nonsingular action $T=\{T_g\}_{g\in G}$ on $X$ such that $T_g\in\text{Aut}_2(X,\mathcal A,\mu)$ and the Poisson suspension $T_*=\{(T_g)_*\}_{g\in G}$ of $T$ is isomorphic to $R$. 
\end{defn}

We establish the following characterization  of  property (T) in terms of nonsingular Poisson suspensions.

\begin{thm}\label{th:T-prop} $G$ has property (T) if and only if each nonsingular Poisson $G$-action admits an absolutely continuous invariant probability measure.
\end{thm}

\begin{proof}
$(\Longrightarrow)$
Let $G$ has property (T) and let $R=\{R_g\}_{g\in G}$ be a nonsingular Poisson $G$-action.
Then there is a  $\sigma$-finite standard measure space $(X,\mathcal A,\mu)$ with $\mu(X)=\infty$ and a $G$-action
$T:G\ni g\mapsto T_g\in\text{Aut}_2(X,\mu)$ such that the Poisson suspension $T_*=\{(T_g)_*\}_{g\in G}$ of $T$
is isomorphic to $R$.
 By the Delorme-Guichardet theorem (see \cite{BHV08Kazh}), $c_T$ is a coboundary.
 Hence there is $f\in L^2(X,\mu)$ such that
\begin{equation}
  c_T(g)=f\circ T_g^{-1}\cdot\sqrt{T_g'}-f.
 \label{bbb}
 \end{equation}
 We define a measure $\nu$ on $X$ be setting $\frac{d\nu}{d\mu}=(1-f)^2$.
 Of course, $\nu(X)=\infty$ and $\nu\in \mathcal{M}_{\mu,2}^+$.
 The equation (\ref{bbb}) yields that $\nu$ is invariant under $T_g$.
$\nu\in\mathcal{M}_{\mu,2}^+$. 
Hence $\nu^*$ is a $T_*$-invariant probability absolutely continuous with respect to $\mu^*$.

$(\Longleftarrow)$ Suppose now that $G$ does not have property (T).
Then there is  a weakly mixing measure preserving action
$S=\{S_g\}_{g\in G}$ on a standard probability space $(X, \mathcal{B},m)$ which is not strongly ergodic.
The latter implies that there is a sequence $\{A_n\}_{n=1}^\infty$ of measurable subsets in $X$ such that $m(A_n)=\frac12$ for every $n\in\Bbb N$ and for each compact subset $K\subset G$,
$$
\lim_{n\to\infty}\sup_{g\in K}m(A_n\triangle S_gA_n)=0.
$$
Fix an increasing sequence $\{K_{n}\}_{n=1}^\infty$  of symmetric compact
subsets of $G$ with $1\in K_1$ and  $G=\bigcup_{n=1}^\infty K_{n}$. 
By
passing to a subsequence in $\{A_n\}_{n=1}^\infty$,  we can assume  without loss of generality that for all $n\in\mathbb{N}$
and $g\in K_{n}$, 
\begin{equation}
\label{eq: almost-invariant A_n}
m\left(A_{n}\triangle S_gA_{n}\right)\leq\frac{2^{-n^2}}{n}.
\end{equation}
Let $\left(\Omega,\mathcal{C},\mathbb{P}\right):=\left(X^{\mathbb{N}},\mathcal{B}^{\mathbb{N}},m^{\mathbb{N}}\right)$.
We will consider the diagonal (measure preserving) action 
$$
T:G\ni g \mapsto T_g:=S_g\times S_g\times\cdots\in\text{Aut}(\Omega,\mathcal C,\mathbb{P})
$$
of $G$ on $\Omega$.
Since $S$ is weakly mixing, so is $T$.
For each $n\in\Bbb N$, we let
\[
B_{n}:=\left\{ x=(x_j)_{j=1}^\infty\in\Omega:\ \forall j\in\left[\frac{n(n-1)}{2},\frac{n(n+1)}{2}\right),\ x_{j}\in A_{n}\right\}.
\]
Of course,  $B_n\in\mathcal{C}$ and $\mathbb{P}\left(B_{n}\right)=2^{-n}$.
Now we define a function 
 $F:\Omega\to[1,+\infty)$ by setting
\[
F(x)=\sqrt{1+\sum_{n=1}^{\infty}2^{n}1_{B_{n}}(x)}.
\]
Since $\sum_{n=1}^\infty\mathbb{P}(B_n)<\infty$, it follows from the 
Borel-Cantelli Lemma that $F(x)<\infty$ at a.e. $x\in\Omega$. 
We note that $F\notin L^{2}(\mathbb{P})$. For $D\in \mathcal{B}$ and $k\in \mathbb{Z}$ we write $[D]_k:=\{x\in \Omega: x_k\in D\}$.

{\sl Claim.}   
We claim that $\left\Vert F^2-F^2\circ T_g\right\Vert _{1}<\infty$ for each $g\in G$.
To prove this inequality,  we choose $N\in\mathbb{N}$ such that 
$g^{-1}\in K_{n}$ for all $n\geq N$.
Writing $I_n=\left[\frac{n(n-1)}{2},\frac{n(n+1)}{2}\right)$, we have 
\[
B_n\triangle T_g^{-1}B_n\subset \bigcup_{k\in I_n} \left[A_n\triangle S_g^{-1}A_n\right]_k.
\]
This together with \eqref{eq: almost-invariant A_n} implies that
\begin{align}\label{eq: bound on B translate}
\left\Vert 1_{B_{n}}-1_{B_{n}}\circ T_g\right\Vert _{1} & =\mathbb{P}\left(B_n\triangle T_g^{-1}B_n\right)\\
&\leq \sum_{k\in I_n} \mathbb{P}\left(\left[A_n\triangle S_g^{-1}A_n\right]_k\right)\nonumber\\
& =n\cdot m\left(A_{n}\triangle S_gA_{n}\right)\leq2^{-n^2}
  \nonumber
\end{align}
for all $n\geq N$.
Hence
\begin{align*}
\left\Vert F^2-F^2\circ T_g\right\Vert _{1} & \leq \sum_{n=1}^{\infty}2^{n}\left\Vert 1_{B_{n}}-1_{B_{n}}\circ T_g\right\Vert _{1}\\
 & \leq\sum_{n=1}^{N-1}2^{n}+\sum_{n=N}^{\infty}2^{n}2^{-n^2}<\infty,
\end{align*}
as claimed.

Now we  define a new measure $\mu$ on $(\Omega,\mathcal{C})$ by setting:
$\mu\sim \mathbb{P}$ and $\frac{d\mu}{d\mathbb{P}}:=F^{2}$.
Then $\mu$ is $\sigma$-finite and $\mu(X)=\infty$.
For each $g\in G$,
\begin{equation} \label{eq: It is in Aut1}
\int_{X}\left|\frac{d\mu\circ T_g}{d\mu}-1\right|d\mu
=\int_{X}\left|F^2\circ T_{g}-F^2\right|d\mathbb{P}
<\infty
\end{equation}
in view of Claim. 
Thus,  $T_{g}\in{\rm Aut}_{1}(\Omega,\mathcal C,\mu)$.
Hence the nonsingular Poisson $G$-action $T_*=\{(T_g)_*\}_{g\in G}$ is well defined
on $(\Omega^*,\mathcal C^*,\mu^*)$.
If $T_*$ had an invariant $\mu^*$-absolutely continuous probability measure then by
Proposition~\ref{prop:coboundary} there would exist an infinite $T$-invariant measure 
$\nu\in\mathcal{M}_{\mu,2}^+$ action.
This contradicts to the fact that $T$ is ergodic and has an equivalent invariant  probability measure.

\end{proof}

We are interested now in the computation of the Furstenberg entropy  of Poisson $G$-actions.
For that, we first  prove the following proposition.

\begin{prop}\label{prop:estimlog}
Let $(\Omega,\mu,T)$ be as in the proof of Theorem~\ref{th:T-prop}.
Then
 $\log T_g'\in L^1(\Omega,\mathbb{P})$ for each $g\in G$.
\end{prop}
\begin{proof}
We first let $B_0:=\Omega$  and define a sequence $\{E_n\}_{n=0}^\infty$ of subsets of $\Omega$ by setting 
$
E_n:=B_n\setminus \left(\bigcup_{k=n+1}^\infty B_k\right).
$
It follows from the definition of $F^2$ that
\[
E_n=\left\{x\in\Omega:\ 2^n<F(x)^2\leq 2^{n+1}\right\}.
\]
Hence $E_0,E_1,\dots$ form 
a countable partition of $\Omega$.
Fix $g\in G$ and note that 
\[
\int_\Omega \left|\log\frac{d\mu\circ T_g}{d\mu}\right|d\mu
= \sum_{n=0}^\infty\int_{E_n} F^2\left|\log\left(\frac{F^2\circ T_g}{F^2}\right)\right|d\mathbb{P}.
\]
We will show that the right hand side of this equality is finite.
If for some $n\in\mathbb{N}$ and $x\in\Omega$, we have that $x\in D_{n,g}:=E_n\cap T_g^{-1}\left(\bigsqcup_{k=n}^\infty E_k\right)$ then $F( T_gx)^2 >2^{n}\geq F(x)^2/2$. 
Since  $|\log y|\leq 2|y-1|$ for each $y>1/2$,
 we see that 
\begin{align*}
 \sum_{n=0}^\infty \int_{D_{n,g}}F^2\left|\log\left(\frac{F^2\circ T_g}{F^2}\right)\right|d\mathbb{P} &\leq   \sum_{n=0}^\infty 2\int_{D_{n,g}}F^2\left|\frac{F^2\circ T_g}{F^2}-1\right|d\mathbb{P}\\
&\leq2\sum_{n=0}^\infty \int_{E_n}\left|F^2\circ T_g-F^2\right|d\mathbb{P}\\
&= 2\left\Vert F^2-F^2\circ T_g\right\Vert _{1}<\infty
\end{align*}
in view of \eqref{eq: It is in Aut1}.
Now if $x\in E_n \setminus T_g^{-1}\left(\bigsqcup_{k=n}^\infty E_k\right)$ then
\[
F^2(x)\left|\log\left(\frac{F(T_gx)^2}{F(x)^2}\right)\right|\leq 2^{n+1}\left|\log\left(2^{-(n+1)}\right)\right|. 
\]
As  $E_n\subset B_n$ and $\bigsqcup_{k=n}^\infty E_k\supset B_n$,
we obtain that
$E_n \setminus T_g^{-1}\left(\bigsqcup_{k=n}^\infty E_k\right)
\subset B_n\triangle T_g^{-1}B_n$.
Let  $N$ be  the smallest natural number such that $g^{-1}\in K_n$ for all $n\geq N$ then,
\begin{align*}
 \sum_{n=0}^\infty \int_{E_n\setminus D_{n,g}}F^2\left|\log\left(\frac{F^2\circ T_g}{F^2}\right)\right|d\mathbb{P} 
 &\leq  \sum_{n=0}^\infty (n+1)2^{n+1}\int_{B_n\triangle T_g^{-1}B_n}d\mathbb{P} \\
&\leq \sum_{n=1}^{N-1}(n+1)2^{n+1}+\sum_{n=N}^{\infty}(n+1)2^{n+1}2^{-n^2},
\end{align*}
where the latter bound follows from \eqref{eq: bound on B translate}. We conclude that 
\[
\int_\Omega \left|\log\frac{d\mu\circ T_g}{d\mu}\right|d\mu<\infty,
\]
as desired.
\end{proof}

We can now prove the following result.

\begin{thm}\label{th:Fur-entr}
 Let $G$ do not have property (T) and let $\kappa$ be a probability measure on $G$.
Then there is a nonsingular $G$-action $T$ on an infinite measure space $(\Omega,\mu)$
such that the Poisson suspension $T_*$ of $T$ is $\mu^*$-nonsingular and 
$\{h_\kappa(T_*,\mu_t^*): t\in(0,+\infty)\}=(0,+\infty)$.
\end{thm}

\begin{proof}
For each $N>0$, we let
$$
C_N:=\sum_{n=1}^{N-1}(n+3)2^{n+1}+\sum_{n=N}^{\infty}(n+2)2^{n+1}2^{-n^2}.
$$
Then we choose an increasing sequence $(K_n)_{n=1}^\infty$ of compacts in $G$ such that
$\bigcup_{n=1}^\infty K_n=G$ and 
$\kappa(K_{n-1}^{-1})>1-\frac 1{2^nC_n}$ for each $n$.
By $K_n^{-1}$ we mean the subset $\{k^{-1}: k\in K_n\}$.
Using this sequence, we construct the dynamical system $(\Omega,\mu, T)$
exactly as in the proof of Theorem~\ref{th:T-prop}.
It follows from the proof of Proposition~\ref{prop:estimlog} that
$$
\sup_{g^{-1}\in K_N}\left|\int_\Omega \log\frac{d\mu\circ T_g}{d\mu}d\mu\right|< C_N
$$
for each $N>0$.
Since $T_g$ preserves a probability measure equivalent to $\mu$, 
it follows  that $T_g$ is conservative  for each $g\in G$.
Therefore, by Proposition~\ref{prop:entr},
$$
\begin{aligned}
h_\kappa(T_*,\mu^*)
&=\sum_{N=1}^\infty\int_{K_N^{-1}\setminus K_{N-1}^{-1}}\left(-\int_\Omega\log\frac{d\mu\circ T_g}{d\mu}d\mu\right)\,d\kappa(g)\\
&\le
\sum_{N=1}^\infty C_N\kappa(K_N^{-1}\setminus K_{N-1}^{-1})\\
&\le
\sum_{N=1}^\infty \frac 1{2^N}=1.
\end{aligned}
$$
By Corollary~\ref{cor:entr-cons}, 
$$
\{h_\kappa(T_*,\mu_t^*): t\in(0,+\infty)\}=\{th_\kappa(T_*,\mu^*):0<t<+\infty\}=(0,+\infty).
$$
\end{proof}

\subsection{A remark on ergodicity of  $T_*$}
It is well known that if a transformation $T$ preserves $\mu$ then  $T_*$ is ergodic if and only if there are no $T$-invariant sets of strictly positive finite measure. 
This is  no longer true for general $T\in\text{Aut}_2(X,\mathcal{A},\mu)$.
Moreover, using the techniques developed in this section, we show in the following example that even when $T$ is ergodic, $T_*$ can be non-ergodic.

\begin{prop}
There exists 
an ergodic transformation $T\in \text{{\rm Aut}}_2(X,\mathcal{A},\mu)$ with $\mu(X)=\infty$ such that $T_*$ is totally dissipative. 
\end{prop}
\begin{proof}
Let $X=\{0,1\}^\mathbb{Z}$ with the product metric and $\mathcal{B}$ be the corresponding Borel $\sigma$-algebra and $S$ be the shift on $X$. 
For $k\in\mathbb{N}$, we 
denote by  $\lambda_k$ the distribution on $\{0,1\}$ such that  $\lambda_k(0)=\frac{1}{4^kk^2}$.
We now define a Borel probability $\mu_k$ on $X$ by setting 
$\mu_k:=\lambda_k^{\otimes\mathbb{Z}}$.
 Finally, let $\Omega:=X^\mathbb{N}$, $\mathbb{P}=\bigotimes_{k=1}^\infty \mu_k$ and $T:=S\times S\times\cdots$. 
 Since $T$ is a direct product of mixing (Bernoulli) transformations,
  $T$ is an ergodic $\Bbb P$-preserving transformation of $\Omega$.
Suppose that we have  a function  $F:\Omega\to (0,+\infty)$ such that
  \begin{itemize} 
    \item[$(\ast)$] $F\not\in L^2(\Bbb P)$,
  \item[$(\ast\ast)$]
 $F-F\circ T^n\in L^2(\Bbb P)$ for all $n\in\mathbb{Z}$ and
 \item[$(\ast\ast\ast)$]
$
\sum_{n\in\mathbb{Z}} e^{-\frac{1}{2}\left\|F-F\circ T^n\right\|_2^2}<\infty.
$
\end{itemize}
 Then, as in the proof of 
Theorem~\ref{th:T-prop}, we define a new measure $\mu$ on $(\Omega,\mathcal C)$ by setting
$\mu\sim \Bbb P$ and  $\frac{d\mu}{d\Bbb P}:=F^2$.
We deduce from $(\ast)$ that $\mu(\Omega)=\infty$.
Since
$$
\int_\Omega|\sqrt{T'}-1|^2d\mu=\int_\Omega|F-F\circ T^{-1}|^2d\Bbb P,
$$
it follows that
 $T\in \text{Aut}_2(X^\mathbb{N},\mathcal{B}^\mathbb{N},\mu)$.
Moreover, 
\[
\sum_{n\in\mathbb{Z}}e^{-\frac12\int |\sqrt{ (T^{-n})'}-1|^2\,d\mu}
=\sum_{n\in\mathbb{Z}} e^{-\frac{1}{2}\int|F-F\circ T^n|^2\,d\Bbb P}<\infty.
\] 
Hence $T_*$ is dissipative by Corollary~\ref{co:total dissipat} and the proposition is proved.
Thus, it remains to build such an $F$.
For that purpose, we let  $f_k:=1_{[1]_k}$ for each $k\in\Bbb N$.
As it will be used repeatedly, we note that for each $k\in\mathbb{N}$, $\left(f_k\circ T^j\right)_{j\in\mathbb{Z}}$ is distributed as an iid sequence of random variables.
 The double sequence $\left(f_k\circ T^j\right)_{k\in\mathbb{N}, j\in\mathbb{Z}}$ is a sequence of independent random variables. 

We now define a function $F:\Omega\to [1,\infty)$ by setting
\[
F:=1+\sum_{k=1}^\infty \sum_{j=0}^{2^k-1}2^kf_k\circ T^j.
\]	
Since 
\[
\sum_{k=1}^\infty \sum_{j=0}^{ 2^k-1}\mathbb{P}\left(\{\omega\in\Omega\mid 2^kf_k\circ T^j\neq 0\}\right)=\sum_{k=1}^\infty \frac{1}{2^kk^2}<\infty,
\]	
it follows from the Borel-Cantelli Lemma that $F<\infty$  almost everywhere. 
For each  $n\in\mathbb{N}$, 
\[
F-F\circ T^n=\sum_{k=1}^\infty \left(Y_k-Y_k\circ T^n\right),
\] 
where $Y_k:=\sum_{j=0}^{2^k-1}2^kf_k\circ T^j$. 
From now on, we will use the notation  $\|r\|_2$ for an arbitrary
measurable map $r:\Omega\to\Bbb R$.
Thus, $\|r\|_2:=\int_\Omega|r|^2\,d\Bbb P$ can be infinite.
We now calculate the sequence $\left(\left\|Y_k-Y_k\circ T^n\right\|_2 \right)_{k=1}^\infty$.
We first note that
\begin{equation}\label{eq: IID1.5}
Y_k-Y_k\circ T^n = \sum_{j=0}^{2^k-1}2^k\left(f_k\circ T^j-f_k\circ T^{j+n}\right).
\end{equation}
  Denote by $m$ the unique non-negative integer  with $2^m\leq n<2^{m+1}$. 
  If $k\leq m$, then 
the right hand side of  (\ref{eq: IID1.5}) is a sum of $2^k$ iid random variables with zero mean and variance $\frac{2}{k^2}(1-4^{-k}k^{-2})$.
Indeed, 
\begin{align*}
\|2^k\left(f_k\circ T^j-f_k\circ T^{j+n}\right)\|_2^2&=4^k\left(2\|f_k\|^2-2\langle f_k,f_k\circ T^n\rangle\right)\\
&=2\cdot4^{k}(\Bbb P([1]_k)-2
\Bbb P([1]_k)^2)
\end{align*}
and $\Bbb P([1]_k)=\mu_k(1)=\frac 1{4^kk^2}$.
 We conclude that for all $k\leq m$, 
\begin{equation}\label{eq: IID2}
\left\|Y_k-Y_k\circ T^n\right\|_2^2=\frac{2^{k+1}}{k^2}\bigg(1-\frac1{4^kk^2}\bigg).
\end{equation}
Next, if $k>m$ then 
\[
Y_k-Y_k\circ T^n = \sum_{j=0}^{n-1}2^kf_k-\sum_{j=2^k}^{2^k+n-1}2^kf_k\circ T^j.
\] 
Since the right hand side can be written as a sum of $n$ iid random variables with zero mean and variance $\frac{2}{k^2}(1-4^{-k}k^{-2})$ 
we see that for all $k> m$, 
\begin{equation}\label{eq: IID3}
\left\|Y_k-Y_k\circ T^n\right\|_2^2=\frac{2n}{k^2}\bigg(1-\frac1{4^kk^2}\bigg).
\end{equation}
Combining \eqref{eq: IID2} and \eqref{eq: IID3} we see that 
\[
\sum_{k=1}^\infty \left\|Y_k-Y_k\circ T^n\right\|_2^2=\sum_{k=1}^{m}\frac{2^{k+1}}{k^2}\bigg(1-\frac1{4^kk^2}\bigg)+n\sum_{k=m+1}^\infty\frac{2}{k^2}\bigg(1-\frac1{4^kk^2}\bigg)<\infty,
\]
where we recall that $m$ in the integer part of $\log_2n$. 
Finally, since $\left(Y_k-Y_k\circ T^n\right)_{k=1}^\infty$ is a sequence of square integrable, zero mean, independent random variables for every $n\in\mathbb{N}$,  it follows that  $F-F\circ T^n\in L^2(\Bbb P)$
and
\[
\left\|F-F\circ T^n\right\|_2^2=\sum_{k=1}^\infty \left\|Y_k-Y_k\circ T^n\right\|_2^2.
\]
As $T$ preserves $\mathbb{P}$, it follows that  $F-F\circ T^n\in L^2(\Bbb P)$ for each $n\in\mathbb{Z}$ and
\[
\left\|F-F\circ T^n\right\|_2^2=\left\|F-F\circ T^{|n|}\right\|_2^2.
\]
Thus, $(\ast\ast)$ is proved. 
 The same equality yields that for all $n\in\mathbb{Z}\setminus\{0\}$, 
\[
\left\|F-F\circ T^n\right\|_2^2\ge \frac{3}{4}\sum_{k=1}^{m}\frac{2^{k+1}}{k^2}\ge  C\frac{|n|}{(\log_2(|n|))^2}
\]
for some $C>0$.
 This implies $(\ast\ast\ast)$.
 Since $\left\|F-F\circ T^n\right\|_2^2\to\infty$, it follows that $F\not\in L^2(\Bbb P)$, i.e. $(\ast)$ holds.

\end{proof}

\section{Appendix}

\subsection{Infinitely divisible random variables}

We recall that a real valued random variable $X$ defined on some
probability space $\left(\Omega,\mathcal{F},\mathbb{P}\right)$ is
\emph{infinitely divisible} if the distribution $p$ of $X$ satisfies that,
for each $k\ge1$, there exists a probability distribution $p_{k}$ on $\Bbb R$
such that $p$ is the $k$-th convolution power of $p_{k}$. 
Equivalently,
there exists $\delta,\kappa\ge0$ and a $\sigma$-finite Borel measure $\sigma$ on $\Bbb R$
satisfying the following conditions: $\sigma(\{ 0\})=0$, $\int_{\mathbb{R}}x^{2}\wedge1d\sigma(x)<\infty$
and such that
\[
\mathbb{E}[e^{iaX}]=\exp\left(-\frac{a^{2}\kappa^{2}}{2}+ia\delta+\int_{\mathbb{R}}(e^{iax}-1-iax1_{\{y\in\Bbb R:\,|y|\le 1\}}(x))\,d\sigma(x)\right)
\]
The  measure $\sigma$ is called the \emph{L\'evy measure} of $X$. 
Conversely, each $\sigma$-finite measure
 $\sigma$ such that $\sigma(\{ 0\})=0$ and 
$\int_{\mathbb{R}}x^{2}\wedge1d\sigma<\infty(x)$ is the L\'evy measure
of some infinitely divisible random variable.
For the proof of the following proposition we refer to \cite[page 39, below formula (8.8)]{Sato99LevPro}.

\begin{prop}
\label{prop:Integrability ID RV} 
Let $X$ be infinitely divisible.
Then
$X\in L^1(\mathbb P)$  if and only if $\int_{\mathbb{R}}|x|1_{\{ y\in\Bbb R:\,|y|>1\}}(x)\,d\sigma(x)<\infty$.
In this case the characteristic function of $X$ can be written as 
\[
\mathbb{E}[e^{iaX}]=\exp\left(-\frac{a^{2}\kappa^{2}}{2}+ia\gamma+\int_{\mathbb{R}}(e^{iax}-1-iax)\,d\sigma(x)\right),\quad a\in\Bbb R,
\]
and $\mathbb{E}[X]=\gamma$.
\end{prop}

\subsection{Stochastic integrals against a Poisson measure}\label{stoc}

Let $\left(X^{*},\mathcal{A}^{*},\mu^{*}\right)$ be the  Poisson space over  a base $(X,\mathcal A,\mu)$.
If $f:X\to\mathbb{R}$ is a measurable function  satisfying  $\int_{X}f^{2}\wedge1d\mu<\infty$
then a so-called {\it stochastic integral} $I_{\mu}(f):X^*\to\Bbb R$ is well defined
(see \cite{Maruyama70IDproc}, up to a slight change) by the following formula
\[
I_{\mu}(f)(\omega)=\lim_{\epsilon\to0}\left(\int_{\{x\in X:\,|f(x)|>\epsilon\}}fd\omega-\int_{\{x\in X:\,|f(x)|>\epsilon\}}f1_{\{x\in X:\,|f(x)|\le 1\}}d\mu\right),
\]
where the limit means the convergence in probability. 
It appears that the random variable $I_{\mu}(f)$
is infinitely divisible.
The L\'evy measure of $I_{\mu}(f)$ is 
$(\mu\circ f^{-1})\restriction{\mathbb{R}\setminus\{0\}}$.
The
 characteristic function of $I_{\mu}(f)$ is
\begin{equation}\label{eq:charac}
\mathbb{E}_{\mu^{*}}[e^{iaI_\mu(f)}]=\exp\left(\int_{X}(e^{iaf}-1-iaf1_{\{x\in X:\,|f(x)|\le1\}})d\mu\right).
\end{equation}
for each $a\in\Bbb R$.

\begin{prop}\label{prop: appendix ID}
Let a function $f:X\to \mathbb{R}$ satisfy $\int_{X}f^{2}\wedge1d\mu<\infty$. Then $I_\mu(f)\in L^1(\mu^*)$ if and only if $f1_{\{x\in X:\,|f(x)|>1\}}\in L^1(\mu)$. 
In this case $\mathbb{E}_{\mu^*}(I_\mu(f))=\int_{\{x\in X:\,|f(x)|>1\}}fd\mu$. 
\end{prop}
\begin{proof}
The integrability criteria for $I_\mu(f)$ follows from Proposition \ref{prop:Integrability ID RV} since $(\mu\circ f^{-1})\restriction{\mathbb{R}\setminus\{0\}}$ is the L\'evy measure of $I_\mu(f)$. If the latter happens, we can rewrite (\ref{eq:charac}) as 
\[
\mathbb{E}_{\mu^{*}}[e^{iaI(f)}]=\exp\left(ia\int_{\{x\in X:\,|f(x)|>1\}}fd\mu+\int_{X}(e^{iaf}-1-iaf)\,d\mu\right).
\]	
By Proposition \ref{prop:Integrability ID RV}, $\mathbb{E}_{\mu^*}(I_\mu(f))=\int_{\{x\in X:\,|f(x)|>1\}}f\,d\mu$. 
\end{proof}

\bibliographystyle{plain}
\bibliography{biblioNS}

\end{document}